\documentclass{amsart}

\usepackage{amssymb,amsfonts,amscd,tikz,amsmath,esint}
\usepackage[utf8]{inputenc}
\usepackage[colorlinks,linktocpage]{hyperref}
\hypersetup{linkcolor=[rgb]{0,0,0.715}}
\hypersetup{citecolor=[rgb]{0,0.715,0}}
\hypersetup{
  pdftitle={Non-collapsed RCD spaces with locally bounded covering geometry},
  pdfauthor={Jikang Wang}
}
\usepackage[T1]{fontenc}
\usepackage{scalerel}
\usepackage[usestackEOL]{stackengine}

\newtheorem{theorem}{Theorem}[section]
\newtheorem{corollary}[theorem]{Corollary}

\newtheorem{lemma}[theorem]{Lemma}

\newtheorem*{thm1}{Theorem A}
\newtheorem*{thm2}{Theorem B}
\newtheorem*{thm3}{Theorem C}

\theoremstyle{definition}
\newtheorem{defn}[theorem]{Definition}

\theoremstyle{remark}
\newtheorem{remark}[theorem]{Remark}

\makeatletter
\let\c@equation\c@theorem
\makeatother
\numberwithin{equation}{section}

\title[Non-collapsed RCD spaces]{Non-collapsed RCD spaces with locally bounded covering geometry}
\author{Jikang Wang}
\address{Department of Mathematics, University of California, Berkeley, CA 94720, USA}

\address{Current address: Institute of Geometry and Physics, University of Science and Technology of China, Hefei, Anhui, China, 230026}

\email{wangjikang@ustc.edu.cn}
\thanks{}
\thanks{}

\thanks{}

\begin{document}
\date{}

\begin{abstract}
We study non-collapsed $\mathrm{RCD}(-(N-1),N)$ spaces with locally bounded covering geometry. 

Our first main result generalizes Gromov's almost flat manifold theorem: if the universal cover is non-collapsing and the diameter is sufficiently small, then the space is bi-H\"older homeomorphic to an infranilmanifold. 
In the smooth setting with $\mathrm{Ric} \ge -(N-1)$, this homeomorphism is actually a diffeomorphism. 
As an application of our arguments, we show that Gromov’s almost flat manifold theorem extends to the $\mathrm{RCD}+\mathrm{CBA}$ setting.

Our second main result generalizes Fukaya's fibration theorem. 
Suppose a sequence of non-collapsed $\mathrm{RCD}(-(N-1),N)$ spaces with locally bounded covering geometry converges in the pointed measured Gromov--Hausdorff sense to a smooth manifold of dimension $k < N$. We then construct a topological fiber bundle map from the RCD spaces to the limit manifold. 
The fibers are infranilmanifolds, and the structure group is affine.

Finally, although there exist collapsing Ricci limit spaces with no topological manifold points, we prove that if the sequence additionally satisfies locally bounded covering geometry and the limit space has dimension $k$, then every $k$-regular point is a topological manifold point.
\end{abstract}

\maketitle

\section{Introduction}
An $\mathrm{RCD}(-(N-1),N)$ space $(X,d,\mathfrak{m})$ is non-collapsed if $\mathfrak{m}=\mathcal{H}^N$; in particular, $N$ must be an integer. A sequence of non-collapsed $\mathrm{RCD}(-(N-1),N)$ spaces $(X_i,d_i,\mathcal{H}^N,p_i)$ is said to be collapsing if $\mathcal{H}^N(B_1(p_i)) \to 0$.

In this paper, we study the topology of $\mathrm{RCD}(-(N-1),N)$ spaces $(X,d,\mathcal{H}^N)$ satisfying $(\rho,v)$-bounded covering geometry. 
More precisely, for any $p \in X$, let $\widetilde{B_{\rho}(p)}$ denote the universal cover of the $\rho$-ball $B_{\rho}(p)$, and let $\tilde{p}$ be a lift of $p$. 
We assume that
$$
\mathcal{H}^N\big(B_{\rho/3}(\tilde{p})\big)\ge v
$$
for some fixed constants $\rho,v > 0$.

We first review the topology of collapsing manifolds with bounded sectional curvature. 
One of the most fundamental results in this theory is Gromov's almost flat manifold theorem.

\begin{theorem}[Gromov's almost flat manifold theorem, \cite{Gromov1978,Rong2019,Ruh1982}]\label{almost flat}
Given $n \in \mathbb{N}$, there exist constants $\epsilon(n) > 0$ and $C(n)>0$ such that for any compact $n$-manifold $M$ satisfying
$$
|\sec_{M}|\,\mathrm{diam}(M)^2 < \epsilon(n),
$$ 
the manifold $M$ is diffeomorphic to an infranilmanifold $\mathcal{N}/\Gamma$, where $\mathcal{N}$ is a simply connected $n$-dimensional nilpotent Lie group, $\Gamma$ is a discrete subgroup of $\mathcal{N} \rtimes \mathrm{Aut}(\mathcal{N})$, and 
$$[ \Gamma : \Gamma \cap \mathcal{N}] \le C(n).$$
\end{theorem}

A related result due to Fukaya states that if a sequence of compact $n$-manifolds with uniformly bounded sectional curvature collapses to a lower dimensional smooth manifold, then the collapse occurs along infranilmanifold fibers.

\begin{theorem}[Smooth fibration theorem, \cite{CFG,Fukaya1987}]\label{regular}
Let $\{M_i\}$ be a sequence of $n$-manifolds with $|\sec_{M_i}| \le 1$ that converge in the Gromov--Hausdorff sense to a compact lower-dimensional smooth manifold $K$.
Then for sufficiently large $i$, there exists a smooth fibration
$$
f_i: M_i \to K
$$
whose fibers are infranilmanifolds and whose structure group is affine. 
Moreover, $f_i$ is a Gromov--Hausdorff approximation.
\end{theorem}

More generally, the theories of singular fibrations and nilpotent Killing structures for collapsed manifolds with $|\sec| \le 1$ have been extensively developed; see, for example, \cite{CFG,CG1,CG2,Fukaya1987,Fukaya1988,Fukaya1989}. 
However, Theorems~\ref{almost flat} and~\ref{regular} may fail if the sectional curvature bound is replaced by a Ricci curvature bound; see \cite{Anderson1992}.

\medskip

In 2016, Rong introduced the notion of $n$-dimensional manifolds satisfying the $(\rho,v)$-bounded Ricci covering geometry. 
More precisely, one assumes that $\mathrm{Ric}_M \ge -(n-1)$ and that for any $p \in M$, if $\tilde{p}$ is a lift of $p$ to the universal cover of the $\rho$-ball $B_{\rho}(p)$, then
$$
\mathrm{Vol}(B_{\rho/3}(\tilde{p})) \ge v.
$$
According to \cite{CFG}, any $n$-manifold with $|\sec| \le 1$ satisfies the $(\rho,v)$-bounded covering geometry condition for some constants $\rho,v > 0$ depending only on $n$.

Theorems~\ref{almost flat} and~\ref{regular} can be generalized to manifolds with Ricci curvature bounded below and $(\rho,v)$-bounded covering geometry; see Theorems~\ref{inf} and~\ref{fib}. 
In fact, if $\mathrm{diam}(M) < \rho$, then the $(\rho,v)$-bounded covering geometry is equivalent to the non-collapsing of the universal cover $\widetilde{M}$.

\begin{theorem}[\cite{HKRX2020,Rong2022}] \label{inf}
Given $n,v>0$, there exist constants $\epsilon(n,v) > 0$ and $C(n)>0$ such that if an $n$-manifold $M$ satisfies
$$
\mathrm{Ric} \ge -(n-1), \quad 
\mathrm{diam}(M) \le \epsilon(n,v), \quad 
\mathrm{Vol}(B_1(\tilde{p})) \ge v, \quad \forall \tilde{p} \in \widetilde{M},
$$
then $M$ is diffeomorphic to an infranilmanifold, where the constant $C(n)$ controls the index appearing in the definition of an infranilmanifold.
\end{theorem}

We briefly summarize the two proofs of Theorem~\ref{inf}. 
In \cite{HKRX2020}, it is shown that the Ricci flow on $M$ exists for a definite time. 
Evolving the metric by the Ricci flow produces an almost flat metric on $M$, to which Gromov's almost flat manifold theorem can be applied. 
In contrast, the proof in \cite{Rong2022} uses a successive blow-up argument together with the structure theory of iterated bundles, thereby avoiding reliance on Gromov's theorem.

\medskip

Recently, building on earlier work of Kapovitch and Wilking, Zamora and Zhu proved a topological rigidity result for $\mathrm{RCD}(K,N)$ spaces with small diameter. 
By the Margulis lemma (Theorem~\ref{KW}), if the diameter of a $\mathrm{RCD}(K,N)$ space $(X,p)$ is sufficiently small, then $\pi_1(X,p)$ contains a nilpotent subgroup $G$ of index at most $C(N)$. 
Consequently, one obtains a descending central series
$$
G= G_1 \vartriangleright G_2 \vartriangleright \dots \vartriangleright G_k=\{e\}
$$
such that each quotient $G_i/G_{i+1}$ is cyclic. 
The rank of $\pi_1(X,p)$ is defined as the number of indices $i$ for which $G_i/G_{i+1}$ is infinite.

\begin{theorem}[\cite{ZamoraZhu2024}]\label{ZZ} 
For any $K \in \mathbb{R}$ and $N \ge 1$, there exists $\epsilon > 0$ such that if a $\mathrm{RCD}(K,N)$ space $(X,d,\mathfrak{m},p)$ has diameter less than $\epsilon$, then
$\mathrm{rank}\big(\pi_1(X,p)\big) \le N$. 
Moreover, if $\mathrm{rank}\big(\pi_1(X,p)\big)=N$, then $X$ is homeomorphic to an infranilmanifold of dimension $N$.
\end{theorem}

Zamora and Zhu further conjectured that the homeomorphism in Theorem~\ref{ZZ} can be chosen to be bi-H\"older.

At present, the proofs of Theorem~\ref{inf} and the rigidity part of Theorem~\ref{ZZ} are fundamentally different. 
The proofs of Theorem~\ref{inf} in \cite{HKRX2020,Rong2022} rely on the smooth structure and therefore do not directly extend to the non-smooth setting. 
On the other hand, the proof of the rigidity part of Theorem~\ref{ZZ} in \cite{ZamoraZhu2024} uses topological results for aspherical manifolds to obtain a homeomorphism, but does not establish a diffeomorphism even in the smooth case.

\medskip

The first main result of this paper provides a new approach that applies simultaneously to Theorem~\ref{inf} and to the rigidity part of Theorem~\ref{ZZ}. 
Moreover, we show that the homeomorphism can be chosen to be bi-H\"older, thereby confirming the conjecture of Zamora and Zhu~\cite{ZamoraZhu2024}. 

The main difference between our work and that of~\cite{ZamoraZhu2024} is that we construct the bi-H\"older homeomorphism (or diffeomorphism) explicitly using the $(N,\delta)$-splitting map. We emphasize that our proof is independent of Gromov's almost flat theorem and of the previous results in~\cite{HKRX2020,Rong2022,ZamoraZhu2024}.

\begin{thm1}
Given $N,v>0$, there exist constants $\epsilon(N,v) > 0$ and $C(N)>0$ such that if a $\mathrm{RCD}(-(N-1),N)$ space $(X,d,\mathcal{H}^N)$ satisfies
$$
\mathrm{diam}(X)\le \epsilon \le \epsilon(N,v), \quad
\mathcal{H}^N(B_1(\tilde{p})) \ge v, \quad \forall \tilde{p} \in \widetilde{X},
$$
then $X$ is bi-H\"older homeomorphic to an infranilmanifold $\mathcal{N}/\Gamma$, where $\mathcal{N}$ is a simply connected $N$-dimensional nilpotent Lie group equipped with a left-invariant metric, $\Gamma$ is a discrete subgroup of $\mathcal{N} \rtimes \mathrm{Aut}(\mathcal{N})$, and
$[ \Gamma : \Gamma \cap \mathcal{N}] \le C(N)$.

More precisely, there exists a homeomorphism $f:X \to \mathcal{N}/\Gamma$ such that, for any $x,y \in X$,
$$
(1-\Phi(\epsilon \,|\, N,v))d(x,y)^{1+\Phi(\epsilon \,|\, N,v)} 
\le d\big(f(x),f(y)\big) 
\le (1+\Phi(\epsilon \,|\, N,v))d(x,y),
$$
where $\Phi(\epsilon\mid N,v) \to 0$ as $\epsilon \to 0$.

Moreover, if $X$ is a smooth $N$-manifold with $\mathrm{Ric} \ge -(N-1)$, then $f$ is a diffeomorphism.
\end{thm1}

\begin{remark}
Let $(X,d,\mathcal{H}^N,p)$ be a $\mathrm{RCD}(-(N-1),N)$ space with sufficiently small diameter. 
By \cite{ZamoraZhu2024}, if $\mathrm{rank}\big(\pi_1(X,p)\big) = N$, then the universal cover $\widetilde{X}$ is non-collapsing. 
Conversely, if $\widetilde{X}$ is non-collapsing, then $\mathrm{rank}\big(\pi_1(X,p)\big)=N$ by the argument used in the proof of Theorem A. An analogous observation for Riemannian manifolds was done by Naber-Zhang in \cite{NaberZhang2016}.
\end{remark}

\medskip

Our method in the proof of Theorem~A extends to metric spaces with mixed curvature bounds. 
Kapovitch proved that Gromov’s almost flat manifold theorem holds for weighted closed manifolds with upper sectional curvature and lower Bakry-\'Emery Ricci curvature bounds.

\begin{theorem}[\cite{Kapovitch2021}]\label{kapo}
For any $1<N<\infty$, there exists $\epsilon>0$ such that the following holds. 
If $(M,g,e^{-f}\mathcal{H}^k)$ is a weighted closed Riemannian $n$-manifold with $k \le N$, $\mathrm{sec} \le \epsilon$, $\mathrm{diam} \le \epsilon$, and $\mathrm{Ric}_{f,N} \ge -\epsilon$, then $M$ is diffeomorphic to an infranilmanifold.
\end{theorem}

We briefly recall the proof of Theorem~\ref{kapo}. 
Kapovitch first proved that $M$ is aspherical using a fibration theorem combined with an induction argument. 
It follows that $M$ is homeomorphic to an infranilmanifold. 
Finally, by employing Ricci flow smoothing techniques together with Gromov's almost flat theorem, one concludes that $M$ is diffeomorphic to an infranilmanifold.

In \cite{Kapovitch2021}, it was conjectured that Gromov's almost flat manifold theorem should also hold under the $\mathrm{RCD} + \mathrm{CBA}$ conditions. 
We recall the structure theory of $\mathrm{RCD} + \mathrm{CBA}$ spaces from \cite{KKK}. 
If $(X,d,\mathfrak{m})$ is a $\mathrm{RCD}(K,N)$ space satisfying $\mathrm{CBA}(1)$ with $N<\infty$, then $X$ is a manifold with boundary and $\dim(X) \le N$. 
Moreover, the manifold part of $X$ is a $C^1$-smooth manifold equipped with a $C^0 \cap \mathrm{BV}$ Riemannian metric inducing the distance function $d$. 
In particular, if $\partial X = \emptyset$, then $X$ is a $C^1$-manifold.

By a slight modification of the proof of Theorem A, we prove the following conjecture proposed in \cite{Kapovitch2021}.

\begin{theorem}\label{mixed}
For any $1<N<\infty$, there exists $\epsilon(N)>0$ such that the following holds. 
If $(X,d,\mathfrak{m})$ is a $\mathrm{RCD}(-\epsilon,N)$ space such that $(X,d)$ satisfies $\mathrm{CBA}(\epsilon)$, $\partial X=\emptyset$, and $\mathrm{diam}(X)\le \epsilon$, then $X$ is bi-Lipschitz diffeomorphic to an infranilmanifold of dimension at most $N$. 
That is, there exist an infranilmanifold $\mathcal{N}/\Gamma$ and a diffeomorphism $f:X \to \mathcal{N}/\Gamma$ such that for any $x,y \in X$,
$$
(1-\Phi(\epsilon \,|\, N))d(x,y)
\le d\big(f(x),f(y)\big)
\le (1+\Phi(\epsilon \,|\, N))d(x,y),
$$
where $\Phi(\epsilon \,|\, N) \to 0$ as $\epsilon \to 0$.
\end{theorem}

It was shown in \cite{Kapovitch2021} that Theorem~\ref{mixed} holds up to homeomorphism. However, it remained open whether Ricci flow techniques could be used to smooth the metric in this setting. 
In particular, our proof of Theorem~\ref{mixed} provides an alternative approach to Theorem~\ref{kapo}, avoiding both Ricci flow smoothing techniques and Gromov's almost flat manifold theorem.

\medskip

Our next result establishes a regular fibration theorem in the RCD setting. 
In the smooth category, the following fibration theorem was proved by Huang and Rong.

\begin{theorem}[\cite{Huang2020,Rong2022}] \label{fib}
Given $n,v,\rho>0$, there exists $\epsilon(n,v,\rho) > 0$ such that the following holds. 
Let $M$ be a smooth compact Riemannian $n$-manifold and $K$ a smooth $k$-manifold satisfying
\[
\mathrm{Ric}_M \ge -(n-1), \quad
(\rho,v)\text{-bounded covering geometry holds on } M,\]
\[
|\mathrm{sec}_K| \le 1, \quad
\mathrm{inj}_K \ge 1, \quad
d_{GH}(M,K) < \epsilon < \epsilon(n,v,\rho).
\]
Then there exists a smooth fiber bundle map $f:M \to K$ which is a $\Phi(\epsilon \,|\, n,v,\rho)$-Gromov--Hausdorff approximation, where $\Phi(\epsilon \,|\, n,v,\rho) \to 0$ as $\epsilon \to 0$. 
Moreover, the fibers are infranilmanifolds and the structure group is affine.
\end{theorem}

Huang first constructed the fibration in \cite{Huang2020}, and Rong subsequently proved that the fibers are infranilmanifolds and that the structure group can be reduced to the affine group in \cite{Rong2022}.

We generalize Theorem~\ref{fib} to $\mathrm{RCD}(-(N-1),N)$ spaces $(X,d,\mathcal{H}^N)$ with $(\rho,v)$-bounded covering geometry.

\begin{thm2}
Given $N,v,\rho>0$, suppose that a sequence of compact $\mathrm{RCD}(-(N-1),N)$ spaces $(X_i,d_i,\mathcal{H}^N)$ with $(\rho,v)$-bounded covering geometry converges in the Gromov--Hausdorff sense to a compact smooth Riemannian $k$-manifold $K$. 
Then, for sufficiently large $i$, there exists a fiber bundle map
$$
f_i:X_i \to K
$$
which is an $\epsilon_i$-GHA, with $\epsilon_i \to 0$ as $i \to \infty$. 
Moreover, the fibers are homeomorphic to infranilmanifolds, and the structure group is affine. 
In particular, $X_i$ is a topological manifold for sufficiently large $i$.
\end{thm2}

\medskip

Next, we study the limit spaces of $\mathrm{RCD}(-(N-1),N)$ spaces $(X_i,d_i,\mathcal{H}^N)$ with $(\rho,v)$-bounded covering geometry. 
By \cite{HNW2023,Zhou2024}, for any $k \ge 3$, there exist collapsing Ricci limit spaces of rectifiable dimension $k$ that contain no manifold points. 
We prove that, under the $(\rho,v)$-bounded covering geometry, the set of manifold points in the limit has full measure.

\begin{thm3}
Let $(X_i,d_i,\mathcal{H}^N,p_i)$ be a sequence of pointed $\mathrm{RCD}(-(N-1),N)$ spaces with $(\rho,v)$-bounded covering geometry, and suppose that
$$
(X_i,d_i,\mathcal{H}^N,p_i)
\overset{\mathrm{mGH}}{\longrightarrow}
(X,d,\mathfrak{m},p),
$$
where the limit space $X$ has rectifiable dimension $k$. 
Then:
\begin{enumerate}
\item[\textup{(1)}] 
For any $q \in \mathcal{R}^k(X)$, a neighborhood of $q$ is bi-H\"older homeomorphic to an open ball in $\mathbb{R}^k$;

\item[\textup{(2)}] 
If $q \in \mathcal{R}^k(X)$ and $q_i \in X_i$ converge to $q$, then for sufficiently large $i$, a neighborhood $U_i$ of $q_i$ is bi-H\"older homeomorphic to $V \times \mathcal{N}_i/\Gamma_i$, where $V$ is a neighborhood of $q$, and $\mathcal{N}_i/\Gamma_i$ is an infranilmanifold.
\end{enumerate}
\end{thm3}

\begin{remark}
In the setting of Theorem~C, Rong showed that the Hausdorff dimension of $(X,d)$ equals $k$ and that every tangent cone of $X$ is a metric cone.
\end{remark}

\medskip

We now sketch the proof of Theorem A. 
We argue by contradiction. 
Suppose there exists a sequence of $\mathrm{RCD}(-(N-1),N)$ spaces $(X_i,d_i,\mathcal{H}^N,p_i)$ with non-collapsing universal covers and $\mathrm{diam}(X_i) \to 0$, while none of the $X_i$ are bi-H\"older homeomorphic to an infranilmanifold.

By the Margulis lemma, there exists a nilpotent subgroup $G_i$ of $G_i'=\pi_1(X_i,p_i)$ with index $\le C(N)$. 
After possibly rescaling the metric, we may assume that the following diagram holds:
\begin{center}
$\begin{CD}
(\widetilde{X}_i,\tilde{p}_i,G_i,G_i') @>eGH>> (\mathbb{R}^N,\tilde{p},G,G')\\
@VV\pi V @VV\pi V\\
\widetilde{X}_i/G_i @>GH>> \mathrm{pt}.
\end{CD}$
\end{center}

We can show that the limit group $G$ acts on $\mathbb{R}^N$ by free translations. 
By the structure theorem for approximate groups \cite{BGT} (see also \cite{Zamora2020}), a neighborhood of the identity in $G_i$ forms a nilprogression, i.e., a finite subset encoding all generators and relations of a lattice in a simply connected nilpotent Lie group.

Since $\mathrm{diam}(\widetilde{X}_i/G_i)\to 0$, the group $G_i$ is determined by a small neighborhood of the identity. 
Hence $G_i$ is isomorphic to a lattice in a simply connected $N$-dimensional nilpotent Lie group $\mathcal{N}_i$, whose Lie algebra converges to that of $\mathbb{R}^N$. 
Equipping $\mathcal{N}_i$ with a left-invariant metric, we obtain pointed $C^4$-convergence to the flat $\mathbb{R}^N$.

Relying on the rigidity, we can show that $G_i'$ embeds into $\mathcal{N}_i \rtimes \mathrm{Aut}(\mathcal{N}_i)$ and acts isometrically on $\mathcal{N}_i$. 
Moreover, $(\widetilde{X}_i,\tilde{p}_i,G_i')$ is pointedly equivariantly Gromov--Hausdorff close to $(\mathcal{N}_i,e,G_i')$. 
Using an extension lemma, we construct a global map
$$
h_i:\widetilde{X}_i \to \mathcal{N}_i
$$
that is almost $G_i'$-equivariant and is a Gromov--Hausdorff approximation on every unit ball.

Next, we choose a normal subgroup $G_i'' \unlhd G_i'$ of finite index such that $G_i'' \subset G_i$ and $G_i'' \cap B_{1}(e)=\{e\}$. 
Applying the averaging argument (Theorem~\ref{inv}), we obtain a $G_i'/G_i''$-equivariant map
$$
f_{i}:\widetilde{X}_i/G_i'' \to \mathcal{N}_i/G_i''
$$
which is locally almost $N$-splitting. 
By the canonical Reifenberg method, this map is a bi-H\"older homeomorphism. 
Equivariance then implies that
$
X_i=\widetilde{X}_i/G_i'
$
is bi-H\"older homeomorphic to the infranilmanifold $\mathcal{N}_i/G_i'$, completing the proof of Theorem A.

\medskip

In the smooth case (Theorem~\ref{fib}), the fibration map $f_i$ is constructed by averaging and gluing local almost $k$-splitting maps, yielding a smooth Gromov--Hausdorff approximation. 
Using the local bounded covering geometry, one shows that $df_i$ is non-degenerate, and hence $f_i$ is a fiber bundle by the implicit function theorem.

In the non-smooth setting (Theorem~B), we construct a Gromov--Hausdorff approximation $f_i:X_i \to K$ that is locally almost $k$-splitting. 
Passing to local universal covers, we observe that the collapse occurs precisely along the orbits of the local deck transformation group. 
Applying the argument of Theorem~A to these collapsing directions yields the desired fiber structure. 
The structure group is affine because the nilpotent structure of the fiber is independent of the base point; see also \cite{Rong2022}.

\begin{remark}
Pointed Gromov--Hausdorff approximations generally do not capture global topology in the non-compact setting, as they only control geometry inside large balls. 
This is why we construct a global map $h_i$ that is almost equivariant and serves as a Gromov--Hausdorff approximation on every unit ball.
\end{remark}

{\bf Acknowledgments.}
The author would like to thank  Xiaochun Rong, Vitali Kapovitch, Jiayin Pan, Shicheng Xu, Xingyu Zhu and Sergio Zamora for helpful discussions. The author would also thank anonymous referees for helpful suggestions.

\section{Preliminaries}

\subsection{Equivariant Gromov--Hausdorff convergence and the isometry group of a Ricci limit space}

We recall the notion of equivariant Gromov--Hausdorff convergence introduced by Fukaya and Yamaguchi \cite{Fukaya1986,FukayaYamaguchi1992}.

Let $(X,p)$ and $(Y,q)$ be pointed metric spaces. 
Let $H \subset \mathrm{Isom}(X)$ and $K \subset \mathrm{Isom}(Y)$ be closed subgroups. 
For $r>0$, define
\[
H(p,r)=\{h \in H \mid d(hp,p) \le r \}, 
\qquad 
K(q,r)=\{k \in K \mid d(kq,q) \le r \}.
\]

For $\epsilon>0$, a pointed $\epsilon$-equivariant Gromov--Hausdorff approximation (briefly, an $\epsilon$-eGHA) from $(X,p,H)$ to $(Y,q,K)$ consists of a triple of maps $(f,\phi,\psi)$:
\[
f:B_{\frac{1}{\epsilon}}(p) \to B_{\frac{1}{\epsilon}+\epsilon}(q), 
\quad 
\phi:H\!\left(p,\tfrac{1}{\epsilon}\right) \to K\!\left(q,\tfrac{1}{\epsilon}\right), 
\quad 
\psi:K\!\left(q,\tfrac{1}{\epsilon}\right) \to H\!\left(p,\tfrac{1}{\epsilon}\right),
\]
such that:

\begin{enumerate}
\item[(1)] $f(p)=q$, the image $f\!\left(B_{\frac{1}{\epsilon}}(p)\right)$ is $2\epsilon$-dense in $B_{\frac{1}{\epsilon}+\epsilon}(q)$, and
\[
|d(f(x_1),f(x_2)) - d(x_1,x_2)| \le \epsilon,
\quad \forall x_1,x_2 \in B_{\frac{1}{\epsilon}}(p);
\]

\item[(2)] $
d\bigl(\phi(h)f(x),\, f(hx)\bigr) < \epsilon,
\quad \forall h \in H\!\left(p,\tfrac{1}{\epsilon}\right), \ 
x \in B_{\frac{1}{\epsilon}}(p);
$

\item[(3)] 
$
d\bigl(kf(x),\, f(\psi(k)x)\bigr) < \epsilon,
\quad \forall k \in K\!\left(q,\tfrac{1}{\epsilon}\right), \ 
x \in B_{\frac{1}{\epsilon}}(p).
$
\end{enumerate}

The equivariant Gromov--Hausdorff distance
\[
d_{eGH}\bigl((X,p,H),(Y,q,K)\bigr)
\]
is defined as the infimum of all $\epsilon>0$ such that there exists an $\epsilon$-eGHA. 
A sequence $(X_i,p_i,G_i)$ is said to converge equivariantly to $(X,p,G)$ if
\[
d_{eGH}\bigl((X_i,p_i,G_i),(X,p,G)\bigr) \to 0.
\]

Given a Gromov--Hausdorff approximation $f$ as in condition (1), one can construct an admissible metric on the disjoint union
\[
B_{\frac{1}{\epsilon}}(p) \sqcup B_{\frac{1}{\epsilon}+\epsilon}(q)
\]
such that the canonical inclusions
\[
B_{\frac{1}{\epsilon}}(p) \hookrightarrow 
B_{\frac{1}{\epsilon}}(p) \sqcup B_{\frac{1}{\epsilon}+\epsilon}(q),
\qquad
B_{\frac{1}{\epsilon}+\epsilon}(q) \hookrightarrow 
B_{\frac{1}{\epsilon}}(p) \sqcup B_{\frac{1}{\epsilon}+\epsilon}(q)
\]
are isometric embeddings and
\[
d(x,f(x)) \le 2\epsilon, 
\quad \forall x \in B_{\frac{1}{\epsilon}}(p).
\]
Throughout, whenever the (equivariant) Gromov--Hausdorff distance between two spaces is sufficiently small, we assume that such an admissible metric has been fixed.

The following theorem gives precompactness for equivariant Gromov--Hausdorff convergence.

\begin{theorem}[\cite{Fukaya1986,FukayaYamaguchi1992}] \label{eGH}
Let $(X_i,p_i)$ be a sequence of pointed metric spaces converging to a limit space $(X,p)$ in the pointed Gromov--Hausdorff sense. 
For each $i$, let $G_i$ be a closed subgroup of $\mathrm{Isom}(X_i)$. 
Then, after passing to a subsequence if necessary,
\[
(X_i,p_i,G_i)
\overset{eGH}{\longrightarrow}
(X,p,G),
\]
where $G$ is a closed subgroup of $\mathrm{Isom}(X)$. 
Moreover, the quotient spaces $(X_i/G_i,\bar p_i)$ converge in the pointed Gromov--Hausdorff sense to $(X/G,\bar p)$.
\end{theorem}



\subsection{Geometric theory of RCD$(K,N)$ spaces}

In this subsection, we review the structure theory of $\mathrm{RCD}(K,N)$ spaces. We refer the reader to the standard literature \cite{AGS2014,EKS2015,LottVillani} for the basic definitions and properties of RCD spaces. 
An $\mathrm{RCD}(K,N)$ space $(X,d,\mathfrak{m})$ is non-collapsed if $N \in \mathbb{N}$ and $\mathfrak{m}=\mathcal{H}^N$.

We begin by recalling the notion of regular points. We use $0^k$ to denote the origin in $\mathbb{R}^k$.

\begin{defn}
Let $(X,d,\mathfrak{m})$ be an $\mathrm{RCD}(K,N)$ space. Given $\epsilon>0$, $r>0$, and $k \in \mathbb{N}$, define
\[
\mathcal{R}_{\epsilon,r}^k(X)
=
\left\{
x \in X :
d_{GH}\big(B_s(x), B_s(0^k)\big) < \epsilon s,
\quad \forall 0<s<2r
\right\}.
\]
Set
\[
\mathcal{R}_\epsilon^k(X)
=
\bigcup_{r>0}\mathcal{R}_{\epsilon,r}^k(X),
\qquad
\mathcal{R}^k(X)
=
\bigcap_{\epsilon>0}\mathcal{R}_\epsilon^k(X).
\]
\end{defn}

For any $\mathrm{RCD}(K,N)$ space $(X,d,\mathfrak{m})$, there exists $k \le N$ such that 
\[
\mathfrak{m}\big(X \setminus \mathcal{R}^k(X)\big)=0;
\]
see \cite{BS2023}. This integer $k$ is called the rectifiable dimension of $(X,d,\mathfrak{m})$.

\medskip

We next recall the theory of almost splitting for RCD spaces, following the notation of \cite{BNS}.

\begin{defn}
Let $(X,d,\mathfrak{m})$ be an $\mathrm{RCD}(K,N)$ space with $K \in \mathbb{R}$ and $1 \le N < \infty$. 
Let $p \in X$ and $s>0$. 
A map $u: B_{2s}(p) \to \mathbb{R}^k$ is called a $(k,\delta)$\emph{-splitting map} if it belongs to the domain of the local Laplacian on $B_{2s}(p)$ and satisfies
\[
|\nabla u|_{B_s(p)} \le C(N),
\]
\[
\sum_{a,b=1}^k 
\fint_{B_s(p)}
\big|
\langle \nabla u^a,\nabla u^b \rangle - \delta_a^b
\big|
\le \delta^2,
\]
\[
\sum_{a=1}^k 
\fint_{B_s(p)}
s^2|\mathrm{Hess}(u^a)|^2
\le \delta^2.
\]
\end{defn}

Such a map $u$ is called an \emph{almost $k$-splitting map} if it is $(k,\delta)$-splitting for sufficiently small $\delta>0$. 
In much of the literature, one additionally assumes that $u$ is harmonic; however, we will not impose harmonicity in this paper. 
By Lemma~4.4 of \cite{HondaPeng2024}, assuming further that
\[
\fint_{B_{2s}(p)} |\nabla \Delta u| \le \delta,
\]
then 
\[
|\nabla u|_{B_s(p)} \le 1 + \Phi(\delta \mid N,k),
\]
where $\Phi(\delta \mid N,k) \to 0$ as $\delta \to 0$.

It is classical that the existence of an almost splitting map is equivalent to pointed measured Gromov--Hausdorff closeness to a space that splits off a Euclidean factor; see \cite{BNS,CheegerColding1997,CheegerColding2000a,CheegerNaber2015}.

We will also use that the isometry group of an $\mathrm{RCD}(K,N)$ space is a Lie group; see \cite{GuSan2019,Sosa2018}.

We recall the volume convergence theorem for non-collapsed $\mathrm{RCD}(K,N)$ spaces.

\begin{theorem}[Volume convergence, \cite{DeGigli2019}]
Let $(X_i,d_i,\mathcal{H}^N,p_i)$ be a sequence of non-collapsed $\mathrm{RCD}(K,N)$ spaces converging to $(X,d,\mathfrak{m},p)$ in the pointed measured Gromov--Hausdorff sense. 
If $\mathcal{H}^N(B_1(p_i)) \ge v > 0$, then $\mathfrak{m}=\mathcal{H}^N$, and, for every $r>0$,
\[
\mathcal{H}^N(B_r(p_i)) \to \mathcal{H}^N(B_r(p)).
\]
\end{theorem}

By \cite{DSZZ2023}, the Margulis lemma holds for $\mathrm{RCD}(K,N)$ spaces.

\begin{theorem}[Margulis lemma, \cite{DSZZ2023}] \label{KW}
For any $K \in \mathbb{R}$ and $N \ge 1$, there exist constants $C$ and $\epsilon_0>0$ such that for any $\mathrm{RCD}(K,N)$ space $(X,d,\mathfrak{m},p)$ with rectifiable dimension $k$, the image of the natural homomorphism
\[
\pi_1(B_{\epsilon_0}(p),p) \to \pi_1(X,p)
\]
contains a normal nilpotent subgroup whose index is at most $C$. Moreover, this nilpotent subgroup admits a nilpotent basis of length at most $k$.
\end{theorem}

It is currently unknown whether the generalized Margulis lemma holds in this setting, namely, whether the image of
\[
\pi_1(B_{\epsilon_0}(p),p) \to \pi_1(B_1(p),p)
\]
contains a nilpotent subgroup of uniformly bounded index; compare \cite{KapovitchWilking2011}. 
However, we prove the following special case for non-collapsed $\mathrm{RCD}(-(N-1),N)$ spaces with locally bounded covering geometry.

\begin{lemma} \label{KW2}
For any $N \ge 1$, there exist constants $C$ and $\epsilon_0>0$ such that for any $\mathrm{RCD}(-(N-1),N)$ space $(X,d,\mathcal{H}^N,p)$ with $(1,v)$-bounded covering geometry, the image of the natural homomorphism
\[
G= \mathrm{Im}\big(\pi_1(B_{\epsilon_0}(p),p) \to \pi_1(B_1(p),p)\big)
\]
contains a normal nilpotent subgroup of index at most $C$. Moreover, this subgroup admits a nilpotent basis of length at most $N$.
\end{lemma}

\begin{proof}
Let $\widetilde{Y}$ be a connected component of the preimage of $B_{\epsilon_0}(p)$ in the universal cover of $B_1(p)$, and let $\widetilde{p}$ be a lift of $p$. 
Then $G$ acts on $(\widetilde{Y},\widetilde{p})$. 

By Corollaries 11.2 and 11.17 of \cite{BGT}, the conclusion holds modulo a small subgroup $G(\delta)$ defined by
\[
G(\delta)
=
\{ g \in G : d(gx,x)<\delta, \quad \forall x \in B_{1/\delta}(\widetilde{p}) \}.
\]
Since the local universal cover is non-collapsed, volume convergence implies that $G(\delta)$ is trivial for any $\delta<\delta(N,v)$; see the proof of Lemma~\ref{no small subgroup}. 
Therefore, the result holds for $G$.
\end{proof}

We now recall the transformation theorem from \cite{BNS}; see also \cite{CheegerNaberJiang2021,CheegerNaber2015,HondaPeng2024}.

\begin{theorem}[Transformation theorem, \cite{BNS}] \label{transformation}
Let $1 \le N < \infty$. For every $\delta>0$ there exists $\epsilon(N,\delta)>0$ such that the following holds. 
Let $(X,d,\mathcal{H}^N)$ be an $\mathrm{RCD}(-\epsilon^2(N-1),N)$ space and suppose $x \in X$. 
If $B_s(x)$ is $\epsilon^2 s$-close to $B_s(0^N) \subset \mathbb{R}^N$ for any scales $s$ with $r_0 \le s \le 1$, and if $u:B_2(x)\to\mathbb{R}^N$ is an $(N,\epsilon)$-splitting map, then for each scale $s$ with $r_0 \le s \le 1$ there exists a lower triangular matrix $T_s \in \mathbb{R}^{N\times N}$ such that:
\begin{enumerate}
\item $T_s u$ is an $(N,\delta)$-splitting map on $B_s(x)$;
\item 
\[
\fint_{B_s(x)} 
\nabla (T_su)^a \cdot \nabla (T_su)^b \, d\mathcal{H}^N
=
\delta_a^b;
\]
\item $|T_s T_{2s}^{-1} - \mathrm{Id}| \le \delta$.
\end{enumerate}
\end{theorem}

\medskip

An important application of the transformation theorem is the construction of bi-H\"older parametrizations via the canonical Reifenberg method \cite{CheegerNaberJiang2021,HondaPeng2024}.

\begin{theorem}[Canonical Reifenberg method, \cite{CheegerNaberJiang2021,HondaPeng2024}] \label{Rei}
Let $(X,d,\mathcal{H}^N,p)$ be an $\mathrm{RCD}(-\epsilon^2(N-1),N)$ space, and let $u:B_2(p)\to\mathbb{R}^N$ be a harmonic $(N,\epsilon)$-splitting map. 
Then for any $x,y \in B_1(p)$,
\[
(1-\Phi(\epsilon \,|\, N)) d(x,y)^{1+\Phi(\epsilon \,|\, N)}
\le
d(u(x),u(y))
\le
(1+\Phi(\epsilon \,|\, N)) d(x,y).
\]
Moreover, if $X$ is a smooth $N$-manifold with $\mathrm{Ric}\ge -\epsilon^2(N-1)$, then for every $x \in B_1(p)$ the differential $du:T_xX\to\mathbb{R}^N$ is nondegenerate.
\end{theorem}

In Theorem~\ref{Rei}, the harmonicity assumption on $u$ can be replaced by the condition
\[
\fint_{B_s(p)} |\nabla \Delta u| \le \epsilon
\]
for any relevant scales $s$; see \cite{HondaPeng2024}.

\subsection{The universal covers and local relative covers of RCD$(K,N)$ spaces}\label{2.4}

For any $\mathrm{RCD}(K,N)$ space $(X,d,\mathfrak{m})$, the universal cover $\widetilde{X}$ exists and is itself an $\mathrm{RCD}(K,N)$ space equipped with the induced metric and measure \cite{MondinoWei2019}. Moreover, $X$ is semi-locally simply connected \cite{Wang2022,Wang2021}; hence its fundamental group is naturally isomorphic to the deck transformation group of $\widetilde{X}$.

Let $(X,d,\mathfrak{m},p)$ be a compact $\mathrm{RCD}(K,N)$ space with $\mathrm{diam}(X) < D$. 
Let $(\widetilde{X},\tilde{p})$ denote its universal cover, and let $G=\pi_1(X,p)$ be the deck transformation group. Define
\[
G(\tilde{p},20D)
=
\{ g \in G : d(\tilde{p},g\tilde{p}) \le 20D \}.
\]
We recall how $\overline{B}_{10D}(\tilde{p})$ and $G(\tilde{p},20D)$ determine $\widetilde{X}$ and $G$; see \cite{FukayaYamaguchi1992,SantosZamora,Wang2023}.

The set $G(\tilde{p},20D)$ forms a pseudo-group: for some $g_1,g_2 \in G(\tilde{p},20D)$, the product $g_1g_2$ may fail to lie in $G(\tilde{p},20D)$. 
To address this, we define its \emph{groupification} $\hat{G}$ as follows.

Let $F$ be the free group generated by symbols $e_g$ for each $g \in G(\tilde{p},20D)$. 
Let $N$ be the normal subgroup of $F$ generated by all elements
\[
e_{g_1}e_{g_2}e_{g_1g_2}^{-1},
\]
whenever $g_1,g_2 \in G(\tilde{p},20D)$ and their product satisfies $g_1g_2 \in G(\tilde{p},20D)$. 
Define
\[
\hat{G} := F/N.
\]

There is a natural pseudo-group homomorphism
\[
i: G(\tilde{p},20D) \to \hat{G},
\qquad
i(g) = [e_g],
\]
where $[e_g]$ denotes the equivalence class of $e_g$ in $\hat{G}$. 
Define
\[
\pi : \hat{G} \to G,
\qquad
\pi([e_{g_1}\cdots e_{g_k}]) = g_1\cdots g_k.
\]
Then $\pi \circ i$ is the identity on $G(\tilde{p},20D)$, hence $i$ is injective. Since $G(\tilde{p},20D)$ generates $G$, the map $\pi$ is surjective.

We now construct a space by gluing:
\[
\hat{G} \times_{G(\tilde{p},20D)} \overline{B}_{10D}(\tilde{p})
=
\hat{G} \times \overline{B}_{10D}(\tilde{p}) \big/ \sim,
\]
where the equivalence relation is defined by
\[
(\hat{g}i(g),x) \sim (\hat{g},gx)
\]
for any $g \in G(\tilde{p},20D)$, $\hat{g} \in \hat{G}$, and $x \in \overline{B}_{10D}(\tilde{p})$ such that $gx \in \overline{B}_{10D}(\tilde{p})$. 
Equip this quotient with the induced length metric.

Then $\hat{G} \times_{G(\tilde{p},20D)} \overline{B}_{10D}(\tilde{p})$ is a covering space of $\widetilde{X}$ with deck transformation group $\ker(\pi)$. Since $\widetilde{X}$ is simply connected, we obtain:

\begin{theorem}[\cite{FukayaYamaguchi1992,SantosZamora,Wang2023}] \label{local}
$\hat{G}$ is isomorphic to $G$, and 
\[
\hat{G} \times_{G(\tilde{p},20D)} \overline{B}_{10D}(\tilde{p}),
\]
equipped with the induced length metric, is isometric to $\widetilde{X}$.
\end{theorem}

\medskip

We now turn to local relative covers of $\mathrm{RCD}(K,N)$ spaces. 
By \cite{SormaniWei2004}, there exists a sequence of compact manifolds $(M_i,p_i)$ with a uniform lower sectional curvature bound such that the local universal covers of $B_1(p_i)$ admit no convergent subsequence in the pointed Gromov--Hausdorff topology. 
To overcome this difficulty, we use a precompactness theorem for relative covers of open balls due to Xu \cite{Xu2023}.

For $p \in X$ and $r_2 > r_1 > 0$, define the \emph{local relative cover}
\[
(\widetilde{B}(p,r_1,r_2),\tilde{p})
\]
to be the connected component containing $\tilde{p}$ of the preimage of $B_{r_1}(p)$ in the universal cover of $B_{r_2}(p)$.

\begin{theorem}[Precompactness of relative covers, \cite{Xu2023}] \label{precompact}
Equipped with its length metric and measure, the space $(\widetilde{B}(p,r_1,r_2),\tilde{p})$ is globally $A_{K,N,r_1,r_2}(r)$-doubling. That is, there exists a positive non-decreasing function $A_{K,N,r_1,r_2}(r)$ such that
\[
0 < \mathfrak{m}(B_r(x))
\le
A_{K,N,r_1,r_2}(r)\,
\mathfrak{m}(B_{r/2}(x))
\]
for any $x \in \widetilde{B}(p,r_1,r_2)$.

In particular, if $G$ denotes the image of the homomorphism
\[
\pi_1(B_{r_1}(p),p) \to \pi_1(B_{r_2}(p),p),
\]
then the family of triples
\[
(\widetilde{B}(p,r_1,r_2),\tilde{p},G)
\]
is precompact in the pointed equivariant Gromov--Hausdorff topology.
\end{theorem}

\begin{remark}
In Theorem~\ref{precompact}, the quotient space 
\[
\widetilde{B}(p,r_1,r_2)/G
\]
is isometric to $B_{r_1}(p)$ equipped with its intrinsic length metric. 
This length metric may differ from the original metric $d$; however, the two metrics coincide on $B_{r_1/3}(p)$. 
For simplicity, we will assume that $B_{r_1}(p)$ is equipped with its intrinsic length metric; otherwise, we restrict our attention to $B_{r_1/3}(p)$ with the induced metric.
\end{remark}

\medskip

We conclude this subsection with some technical lemmas.

\begin{lemma}[\cite{Wilson}] \label{G_0}
Let $G$ be a nilpotent Lie group and let $G_0$ denote its identity component. 
Then for any compact subgroup $K \subset G$, the commutator subgroup $[K,G_0]$ is trivial.
\end{lemma}

\begin{lemma}[Covering lemma, \cite{KapovitchWilking2011}] \label{covering}
There exists $C(N)>0$ such that the following holds. 
Let $(X,d,\mathfrak{m},p)$ be an $\mathrm{RCD}(-(N-1),N)$ space and let $f:X \to \mathbb{R}$ be a nonnegative function. 
Let $\pi:\widetilde{X}\to X$ be the universal cover and $\tilde{p}$ be a lift of $p$. 
Set $\tilde{f}=f\circ\pi$. Then
\[
C(N)^{-1} \fint_{B_{1/3}(p)} f
\le
\fint_{B_1(\tilde{p})} \tilde{f}
\le
C(N) \fint_{B_1(p)} f.
\]
\end{lemma}

\begin{lemma}[Gap lemma, \cite{KapovitchWilking2011}] \label{gap}
Let $(X_i,p_i,G_i)$ be a sequence of length metric spaces such that
\[
(X_i,p_i,G_i)
\overset{eGH}{\longrightarrow}
(X,p,G).
\]
Suppose there exist constants $b>a>0$ such that the subgroup $\langle G(p,r) \rangle$ is independent of $r \in (a,b)$. 
Then there exists $\epsilon_i \to 0$ such that for sufficiently large $i$, the subgroup $\langle G_i(p_i,r) \rangle$ is independent of $r \in (a+\epsilon_i,\, b-\epsilon_i)$.

\end{lemma}

\begin{lemma}\label{normal} Assume that a discrete group $H$ contains a subgroup $H_0$ of index $|H:H_0| \le C$. Then there exists a normal subgroup $H'$ of $H$ such that $|H/H'| \le C!$ and $H' \subset H_0$. \end{lemma}

\subsection{Approximate groups and almost homogeneous spaces}

The main references for this subsection are \cite{BGT,Goldbring,Zamora2020}.

\begin{defn}[\cite{Goldbring}]
A (symmetric) \emph{local group} $G$ is a topological space with identity element $e \in G$, equipped with a global inverse map $()^{-1}: G \to G$ and a partially defined product map $\cdot : \Omega \to G$,
satisfying the following:

\begin{enumerate}
\item $\Omega$ is an open neighborhood of $(G \times \{e\}) \cup (\{e\} \times G)$ in $G \times G$;
\item the maps $x \mapsto x^{-1}$ and $(x,y) \mapsto xy$ are continuous;
\item if $(gh)k$ and $g(hk)$ are well-defined, then $(gh)k = g(hk)$;
\item $eg = ge = g$ for any $g \in G$;
\item $gg^{-1}$ and $g^{-1}g$ are well-defined and equal to $e$.
\end{enumerate}
\end{defn}

If $\Omega = G \times G$, then $G$ is a (global) topological group.

\begin{defn}
Let $G$ be a local group and $g_1,\dots,g_m \in G$.  
We say the product $g_1\cdots g_m$ is \emph{well-defined} if for every $1 \le i \le j \le m$ there exist elements $g_{[i,j]} \in G$ such that:
\[
g_{[k,k]} = g_k,
\]
and whenever $1 \le i \le j < k \le m$, the product
\[
g_{[i,j]} g_{[j+1,k]}
\]
is well-defined and equals $g_{[i,k]}$.
\end{defn}

For subsets $A_1,\dots,A_m \subset G$, we say $A_1\cdots A_m$ is well-defined if $g_1\cdots g_m$ is well-defined for every choice of elements $g_j \in A_j$.

\begin{defn}
A subset $A \subset G$ is called a \emph{multiplicative set} if it is symmetric ($A=A^{-1}$) and $A^{200}$ is well-defined.
\end{defn}

\begin{defn}
Let $A$ be a finite symmetric subset of a multiplicative set and $C \in \mathbb{N}$.  
We say that $A$ is a \emph{$C$-approximate group} if $A^2$ can be covered by at most $C$ left translates of $A$.
\end{defn}

\begin{defn}
Let $A$ be a $C$-approximate group.  
We call $A$ a \emph{strong $C$-approximate group} if there exists a symmetric subset $S \subset A$ such that:

\begin{enumerate}
\item $(\{asa^{-1} : s \in S,\, a \in A\})^{10^3 C^3} \subset A$;
\item if $g, g^2, \dots,g^{1000} \in A^{100}$, then $g \in A$;
\item if $g, g^2, \dots,g^{10^3 C^3} \in A^{100}$, then $g \in S$.
\end{enumerate}
\end{defn}

\medskip

For a multiplicative set $A \subset G$ and $g \in G$, define the \emph{escape norm}
\[
\|g\|_A 
:= 
\inf \left\{ \frac{1}{m+1} \ \Big|  \ e, g, g^2, \dots, g^m \in A \right\}.
\]

\begin{theorem}[Escape norm estimates, \cite{BGT}] \label{escape}
For each $C>0$, there exists $M>0$ such that if $A$ is a strong $C$-approximate group and $g_1, \dots, g_k \in A^{10}$, then:
\begin{enumerate}
\item $\|g_1\cdots g_k\|_A \le M \sum_{j=1}^k \|g_j\|_A$;
\item $\|g_2 g_1 g_2^{-1}\|_A \le 10^3 \|g_1\|_A$;
\item $\|[g_1,g_2]\|_A \le M \|g_1\|_A \|g_2\|_A$.
\end{enumerate}
\end{theorem}

\begin{remark}
Estimate (3) is analogous to the holonomy bounds in Gromov’s proof of the almost flat manifold theorem; see \cite[Chapter 2]{BuserKarcher}.
\end{remark}

\medskip

\begin{defn}
Let $G$ be a local group, $u_1,\dots,u_r \in G$, and $C_1,\dots,C_r>0$.  
The set
\[
P(u_1,\dots,u_r;C_1,\dots,C_r)
\]
consists of all words in the $u_i^{\pm1}$ such that each $u_i$ and $u_i^{-1}$ appears at most $C_i$ times.  

It is called a \emph{nilprogression of rank $r$} if every such word is well-defined in $G$.
\end{defn}

We say that $P$ is in \emph{$C$-normal form} if:

\begin{enumerate}
\item for $1 \le i < j \le r$,
\[
[u_i^{\pm1},u_j^{\pm1}]
\in 
P\!\left(u_{j+1},\dots,u_r;\frac{C C_{j+1}}{C_i C_j},\dots,\frac{C C_r}{C_i C_j}\right);
\]
\item the elements $u_1^{n_1}\cdots u_r^{n_r}$ are distinct whenever
\[
|n_i| \le \frac{C_i}{C};
\]
\item
\[
\frac{1}{C}\prod_{i=1}^r (2\lfloor C_i \rfloor+1)
\le |P|
\le
C \prod_{i=1}^r (2\lfloor C_i \rfloor+1).
\]
\end{enumerate}

Define the thickness by
\[
\mathrm{thick}(P) := \min_i C_i.
\]
The \emph{grid part} $G(P)$ consists of all elements
\[
u_1^{n_1}\cdots u_r^{n_r}
\quad
\text{with }
|n_i| \le C_i/C.
\]

For any $\epsilon > 0$, we denote by $\epsilon P$ the set 
\[
P(u_1,\dots,u_r;\epsilon C_1,\dots,\epsilon C_r).
\]
If $P$ is a nilprogression in $C$-normal form, then $\epsilon P$ is also a nilprogression in $C$-normal form.

\begin{defn} Let $P(u_1,\dots,u_r;C_1,\dots,C_r)$ be a nilprogression in $C$-normal form with $\mathrm{thick}(P) > C$. Set $\Gamma_P$ to be the abstract group generated by $\gamma_1,\dots,\gamma_r$ with the relation 
$$[\gamma_j,\gamma_k]=\gamma_{k+1}^{\beta_{j,k}^{k+1}} \cdots \gamma_{r}^{\beta_{j,k}^{r}}$$ whenever $j<k$, where $$[u_j,u_k]=u_{k+1}^{\beta_{j,k}^{k+1}} \cdots u_{r}^{\beta_{j,k}^{r}}$$ and $|\beta_{j,k}^l| \le \frac{CC_l}{C_jC_k}$. We say that $P$ is good if each element of $\Gamma_P$ has a unique expression of the form $\gamma_1^{n_1} \cdots \gamma_r^{n_r}$ with $n_1,\dots,n_r \in \mathbb{Z}$. \end{defn} 

\medskip
Roughly speaking, a nilprogression in $C$-normal form can be viewed as a neighborhood of a lattice in a connected and simply connected nilpotent Lie group of the same dimension.
 
\begin{theorem}[{\cite[Theorem~2.54]{Zamora2020}}] \label{Malcev} For each $r \in \mathbb{N}$ and $C>0$, there exists $\epsilon >0$ such that the following holds. Let $P(u_1,\dots,u_r;C_1,\dots,C_r)$ be a nilprogression in $C$-normal form. If $\mathrm{thick}(P)$ is sufficiently large depending on $r$ and $C$, then $P$ is good and the map $u_j \to \gamma_j$ extends to a product-preserving embedding from $G(\epsilon P)$ into $\Gamma_P$. Moreover, $\Gamma_P$ is isomorphic to a lattice in an $r$-dimensional simply connected nilpotent Lie group $\mathcal{N}$. \end{theorem}

\begin{remark} During the proof of Theorem A and B, we shall use nilprogressions $P$ satisfying Theorem \ref{Malcev}. To simplify notation, we identify the grid part $G(\epsilon P)$ with $P$, since our construction is actually based on $G(\epsilon P)$. We therefore simply call $\Gamma_P$ in Theorem \ref{Malcev} the groupification of $P$. \end{remark}

\medskip
Assume that $\{Z_i\}$ is a sequence of $\mathrm{RCD}(-(N-1),N)$ spaces, and that $G_i$ acts on $Z_i$ by discrete, isometric, and measure-preserving transformations. Suppose that
\begin{equation}\label{e0}
(Z_i,p_i,G_i,\mathfrak{m}_i)\overset{emGH}{\longrightarrow}(Z,p,G,\mathfrak{m}).
\end{equation}
Then the limit space $(Z,\mathfrak{m})$ is also an $\mathrm{RCD}(-(N-1),N)$ space. In particular, $G$ is a Lie group. We next focus on the equivariant convergence $G_i \to G$. Assume that $\dim G = r$. 

We first note that every sufficiently small neighborhood of the identity
element $e\in G_i$ forms a strong approximate group. For any $\delta > 0$, define
\[
G(\delta)
=
\left\{ g \in G \;\middle|\; d(gx,x) \le \delta \ \text{for all } x \in \overline{B}_{1/\delta}(p) \right\}.
\]
Then there exist constants $C>0$ and sufficiently small
$\epsilon,\delta$, together with a set $A$ satisfying $G(2\epsilon)\subset A \subset G(\delta/2)$,
such that $A$ is a strong $C$-approximate group. Let
$A_i\subset G_i$ be a sequence of subsets converging to $A$. Then, for all
sufficiently large $i$, we have $G_i(\epsilon) \subset A_i \subset G_i(\delta)$.
Moreover, after increasing $C$ if necessary, each $A_i$ is also a strong
$C$-approximate group. To simplify notation, we may identify $A=G(\delta)$ and $A_i=G_i(\delta)$.

We say that the sequence $\{G_i\}$ has \emph{no small subgroups} if the following holds: 
whenever $H_i \le G_i$ is a subgroup generated by $H(p_i,\epsilon_i)$ with $\epsilon_i \to 0$, 
and $H_i$ converges to the identity in $G$ in the eGH sense, then $H_i = \{e\}$ for all sufficiently large $i$.

If a sequence of discrete groups $G_i$ with no small subgroups converges to a Lie group $G$, then we may apply \cite[Theorem 7.1]{Zamora2020} to conclude that a neighborhood of $e \in G_i$ is a nilprogression $P_i$ in $C$-normal form.
We slightly modify the notation in the original statement of \cite[Theorem 7.1]{Zamora2020}, replacing ultralimits with equivariant Gromov--Hausdorff convergence.

\begin{theorem}\cite{BGT,Zamora2020}\label{nil}
If $\{G_i\}$ has no small subgroups, then for sufficiently small $\delta>0$ there exists $\epsilon>0$ such that for sufficiently large $i$, $G_i(\delta)$ contains a nilprogression $P_i$ in $C$-normal form and
$G_i(\epsilon) \subset G(P_i)$.
\end{theorem}

The nilpotent structure of $P_i$ is determined by the escape norm structure of $G_i(\delta)$. This observation will be crucial in proving that the structure group is affine in Theorem~B.

\medskip
Next, recall the structure theory for nilpotent Lie groups and their Lie algebras. 
\begin{defn} Let $\mathfrak{n}$ be a nilpotent Lie algebra. We say that an ordered basis $\{v_1,\dots,v_r\}$ of $\mathfrak{n}$ is a strong Malcev basis if for any $1 \le k \le r$, the vector subspace $J_k$ generated by $\{v_{k+1},\dots,v_r\}$ is an ideal, and $v_k + J_k$ lies in the center of $\mathfrak{n}/J_k$. \end{defn}

 \begin{theorem}\label{str} Let $\mathcal{N}$ be an $r$-dimensional simply connected nilpotent Lie group and $\mathfrak{n}$ be its Lie algebra with a strong Malcev basis $\{v_1,\dots,v_r\}$. Then\begin{enumerate}

\item $\mathrm{exp}_{\mathcal{N}}: \mathfrak{n} \to \mathcal{N}$ is a diffeomorphism; 
\item $\phi: \mathbb{R}^r \to \mathcal{N}$ given by $\phi(x_1,\dots,x_r)=\mathrm{exp}_{\mathcal{N}}(x_1v_1)\cdots\mathrm{exp}_{\mathcal{N}}(x_rv_r)$ is a diffeomorphism; 
\item If we identify $\mathfrak{n}$ with $\mathbb{R}^r$ by the given basis, then $\mathrm{exp}_{\mathcal{N}}^{-1} \circ \phi : \mathfrak{n} \to \mathfrak{n}$ and $\phi^{-1} \circ \mathrm{exp}_{\mathcal{N}}: \mathfrak{n} \to \mathfrak{n}$ are polynomials of degree bounded by a number depending only on $r$.
  \end{enumerate}  \end{theorem}
 
In the diagram \eqref{e0} with no small subgroups, for sufficiently large $i$, by Theorem~\ref{Malcev} and~\ref{nil}, the (grid part of) nilprogression $P_i$ can be identified with a generating set of a lattice in a simply connected nilpotent group $\mathcal{N}_i$. 

Let $\{u_{i,j}\}_{j=1}^r$ be a generating set of $P_i$. Define
\[
g_{i,j} := u_{i,j}^{\lfloor C_j/C \rfloor},
\]
and let $v_{i,j} \in \mathfrak{n}_i$ be elements in the Lie algebra of $\mathcal{N}_i$ such that
\[
\exp_{\mathcal{N}_i}(v_{i,j}) = g_{i,j}, \quad 1 \le j \le r.
\] 
Then $\{ v_{i,j} \mid 1 \le j \le r\}$ is a strong Malcev basis of $\mathfrak{n}_i$. By the choice of $C_j$, we can find $\epsilon'>0$ such that, for any $j$ and all sufficiently large $i$, $g_{i,j}$ does not lie in the $\epsilon'$-neighborhood of the identity. For each $j$, after passing to a subsequence if necessary, we may assume $g_{i,j} \to g_j \in G$. By construction, $g_j$ is non-trivial.

Choose $v_{j}$ in the Lie algebra $\mathfrak{g}$ of $G$ such that $\mathrm{exp}_{G}(v_{j})=g_{j}$. By \cite{BGT,Zamora2020}, $\{v_j \mid 1 \le j \le r \}$ is a strong Malcev basis of $\mathfrak{g}$, the Lie algebra of $G$. For any fixed $i$, since $\{ v_{i,j} \mid 1 \le j \le r\}$ is a strong Malcev basis of $\mathfrak{n}_i$, the map $\phi_i: \mathbb{R}^r \to \mathcal{N}_i$, as in Theorem \ref{str}, is a diffeomorphism. 

Now we identify $(x_1,x_2,\dots,x_r) \in \mathbb{R}^r$ with $\sum_{j=1}^r x_jv_{i,j} \in \mathfrak{n}_i$, and define $$Q_i: \mathbb{R}^r \times \mathbb{R}^r \to \mathbb{R}^r, \quad Q_i(x,y) = \phi_i^{-1}(\phi_i(x) \phi_i(y)).$$ Similarly define $Q$ for $G$ and $\{ v_{j} , 1 \le j \le r\}$. Roughly speaking, $Q_i$ and $Q$ encode the group laws of $\mathcal{N}_i$ and $G$, and, via differentiation, their corresponding Lie algebra structures. 

\begin{theorem}[Lie algebra structure convergence, \cite{Zamora2020}] \label{Q} For sufficiently large $i$, $Q_i$ and $Q$ are polynomials of degree $\le d(r)$ and coefficients of $Q_i$ converge to corresponding coefficients of $Q$. \end{theorem}

\section{Constructing a GHA map which is locally almost splitting}

In this section, we prove Theorem~\ref{local split} and~\ref{inv}, which generalize the main results of \cite{Huang2020} from the smooth setting to an RCD framework. The proofs are technical and will not be used in later sections.

We first recall the main results of \cite{Huang2020}.

\begin{theorem}[smooth fibration, \cite{Huang2020}] \label{fibration}
Given $n,v,\rho>0$, there exist $\epsilon(n,v,\rho) > 0$ and $C(n)>0$ such that the following holds. Let $M$ be a compact $n$-manifold and $K$ be a compact $k$-manifold  satisfying $\mathrm{Ric}_M \ge -(n-1)$, $|\mathrm{sec}_K| \le 1$, $\mathrm{inj}_K \ge 1$ and $d_{GH}(M,K)<\epsilon < \epsilon(n,v,\rho)$. Assume that $(\rho,v)$-bounded covering geometry holds on $M$.
Then there exists a smooth fiber bundle map $f:M \to K$ that is a $\Phi(\epsilon \,|\, n,v,\rho)$-GHA, with $\Phi(\epsilon \,|\, n,v,\rho) \to 0$ as $\epsilon \to 0$.
\end{theorem} 

Assume a group $G$ acts isometrically on two metric spaces $X_1$ and $X_2$ separately. We call a map $h:X_1 \to X_2$ $\epsilon$-almost $G$-equivariant if $$d(h(gx),gh(x)) < \epsilon$$ for any $ x \in X_1, g \in G$. 

\begin{theorem}[stability for compact group actions, \cite{Huang2020}] \label{sta}
There exists $\epsilon(n)>0$ such that the following holds for any $\epsilon < \epsilon(n)$. Assume that $M$ and $K$ are compact $n$-manifolds satisfying $|\mathrm{sec}_K| \le 1$, $\mathrm{inj}_K \ge 1$, $\mathrm{Ric}_M \ge -(n-1)$. The group $G$ acts isometrically on $M$ and $K$ separately, and there exists an $\epsilon$-GHA $h:M \to K$ which is $\epsilon$-almost $G$-equivariant. Then there exists a $G$-equivariant diffeomorphism $f:M \to K$, which is a $\Phi(\epsilon \,|\, n)$-GHA. 
\end{theorem}

We briefly recall the proofs of Theorem \ref{fibration} and \ref{sta}. In Theorem \ref{fibration}, the injectivity radius of $K$ is at least $1$. Then locally one can identify a small ball in $K$ with an open subset of the tangent space of $K$ at some point. Since $M$ is GH-close to $K$, locally one can construct an almost $k$-splitting map from a small neighborhood in $M$ to the tangent space of $K$ at the corresponding point.

To construct a globally defined map $f:M \to K$, one can glue and average these local almost splitting maps using cut-off functions and the center of mass technique. Then one can obtain a smooth GHA $f:M \to K$. Huang shows that, at any point $p \in M$, the differential $df_p$ coincides with the linear average of the differentials of the local almost $k$-splitting maps. Huang then proves that the differential of any almost $k$-splitting map is nondegenerate under the $(\rho,v)$-bounded covering geometry. Thus $df$ is nondegenerate, and $f$ is a fibration map by the implicit function theorem. 

The proof of Theorem \ref{sta} follows a similar approach. Huang constructs a $G$-equivariant map using the center of mass technique and applies the canonical Reifenberg method to show that the map is a diffeomorphism.

\medskip

In the non-smooth RCD setting, the implicit function theorem is unavailable. However, the function constructed in Huang’s approach can still be interpreted as an almost splitting map in the RCD framework.

We recall a sufficient condition for obtaining a locally almost splitting map. The proof is essentially
adapted from the construction of an almost splitting map in \cite{BruePasSemola2023}.
\begin{lemma}\label{Lemma: local splitting}
Let $(X,d,p,\mathfrak{m})$ be an RCD$(-\epsilon^2(N-1),N)$ space and let $f:B_2(p) \to \mathbb{R}^k$ be an $\epsilon$-GHA to its image. Suppose that $f$ satisfies 
$$\fint_{B_2(p)} |\Delta f| \le \epsilon, \quad \fint_{B_2(p)} |\nabla \Delta f| \le \epsilon.$$ Then $f:B_1(p) \to \mathbb{R}^k$ is a $(k,\Phi(\epsilon \,|\, N,k))$-splitting map.
\end{lemma}

For any metric space $(X,d)$ and $p \in X$, $r>0$, we use $rX$ or $(X,rd)$ to denote the rescaled metric. 
 
\begin{theorem}[Almost $k$-splitting] \label{local split}
Assume a sequence of compact $\mathrm{RCD}(-(N-1),N)$ spaces $(X_i,d_i,\mathfrak{m}_i)$ converges to a smooth compact $k$-manifold $(K,g)$ in the Gromov--Hausdorff sense. 
Then for sufficiently large $i$, there exists a continuous GHA $f_i:X_i \to K$ that is locally almost $k$-splitting. That is, for any $\delta > 0$, there exists $\epsilon >0$ such that, for any $p_i \in X_i$ converging to $p \in K$ and sufficiently large $i$,
$$\frac{1}{\epsilon}B_{\sqrt{\epsilon}}(p_i) \overset{f_i}\longrightarrow \frac{1}{\epsilon}B_{\sqrt{\epsilon}}(p) \overset{\mathrm{exp}_p^{-1}}\longrightarrow T_{p}K \cong \mathbb{R}^k$$ is a $(k,\delta)$-splitting $\delta$-GHA map. 
\end{theorem}
\begin{proof}
Let $\epsilon \in (0,1)$ be small. If necessary, rescale the metric by a factor $r \gg 1$. Since GH convergence is preserved under scaling, we may assume that $\mathrm{inj}_K \ge \frac{10}{\epsilon}$ and $X_i$ is $\mathrm{RCD}(-(N-1)\epsilon^2,N)$. Moreover, we assume that for any $p \in K$, the ball $B_{\frac{10}{\epsilon}}(p)$ is $\epsilon^2$-$C^4$-close, via $\mathrm{exp}_p^{-1}$, to its image in $T_p K$ with the flat metric.

Let $\Phi(\epsilon \,|\, N,k)$ denote a function that converges to $0$ as $\epsilon \to 0$, for fixed $k,N$. Its value may vary from line to line.

Fix $i$ sufficiently large so that $d_{GH}(X_i,K) < \epsilon/10$. Our goal is to construct a $\Phi(\epsilon \,|\, N,k)$-GHA $f_i: X_i \to K$ which is a $(k,\Phi(\epsilon \,|\, N,k))$-splitting map on every ball of radius $\frac{1}{\epsilon}$. 

Let $\{p^j\}_{j=1}^J$ be a maximal $\frac{1}{\epsilon}$-separated set in $K$, and for each $j$ choose $p_i^j \in X_i$ such that $d(p_i^j,p^j)<\epsilon$. Let $\Lambda_j$ be a cutoff function on $K$ such that $\Lambda_j \equiv 1$ on $B_{\frac{2}{\epsilon}}(p^j)$ and $\mathrm{supp}(\Lambda_j) \subset B_{\frac{3}{\epsilon}}(p^j)$. We may assume that $|\nabla^l \Lambda_j| \le \Phi(\epsilon \,|\, N,k)$ for $l=1,2,3$; such cutoff functions exist since $|\mathrm{sec}_K| \le \epsilon^2$ . Let $B_{ij}=B_{\frac{4}{\epsilon}}(p_i^j)$, $B_j=B_{\frac{4}{\epsilon}}(p^j)$, and let $B_j^{-1}$ be the preimage of $B_j$ in $T_{p^j}K$ with the flat metric.

$B_{ij}$ is $\epsilon$-GH close to $B_j$, and $B_j$ is $\epsilon^2$-$C^4$-close to $B_j^{-1}$ with the flat metric. After possibly reducing the radius slightly, there is a $\Phi(\epsilon \,|\, N,k)$-GHA
$$
h_j : B_{ij} \longrightarrow B_j,
$$
such that $\mathrm{exp}_{p^j}^{-1} \circ h_j : B_{ij} \to B_j^{-1} \subset \mathbb{R}^k$ is a harmonic $(k,\Phi(\epsilon \,|\, N,k))$-splitting map.

Take the energy function $E : X_i \times K \to \mathbb{R}$ as follows:
$$
E(x,y) = \frac{\sum_{j=1}^J \Lambda_j(h_j(x)) d(h_j(x),y)^2}{\sum_{j=1}^J \Lambda_j(h_j(x))}.
$$
We first show that $E(x,y)$ is well-defined. Since $\Lambda_j(h_j(x))$ vanishes in a neighborhood of the boundary of $B_{ij}$, it can be naturally extended by zero to a well-defined function on the whole space $X_i$. For any $x \in X_i$, choose $j$ such that $p_j \in K$ is $1.5\epsilon$-close to $x$. Then $h_j(x) \in B_{2\epsilon}(p_j)$ and hence $\Lambda_j(h_j(x)) = 1$. Therefore $\sum_{j=1}^J \Lambda_j(h_j(x))$ is positive for any $x \in X_i$.

Since $h_j$ is a GHA, the set $\{h_j(x)\}_{j=1}^J$ (whenever defined) is contained in a $\Phi(\epsilon \,|\, N,k)$-ball for any fixed $x \in X_i$. Let $y' \in K$ be a point close to $x$. Since $B_{10/\epsilon}(y')$ is $\epsilon^2$-close in $C^4$ to Euclidean space, the squared distance is strictly convex on $B_{1/\epsilon}(y')$. Therefore $E(x,\cdot)$ achieves a global minimum at some $y \in B_{\frac{1}{\epsilon}}(y')$, which is the center of mass with respect to $E$. Define $f_i(x)=y$. Then $f_i$ is a $\Phi(\epsilon \,|\, N,k)$-GHA.

We next show that $f_i$ is a $(k,\Phi(\epsilon \,|\, N,k))$-splitting map on any $\frac{1}{\epsilon}$-ball. For any $x_0 \in X_i$, let $y_0=f_i(x_0)$. There exist at most $C(N)$ points $p^j$ contained in $B_{\frac{4}{\epsilon}}(y_0)$, say $j_1,j_2,\dots,j_C$. Then the value of $f_i$ on $B_{\frac{1}{\epsilon}}(x_0)$ depends only on $h_j$ and $\Lambda_j$ for $j=j_1,j_2,\dots,j_C$.

Define a subset $S \subset \prod_{l=1}^C B_{j_l}$ by
\[
S:=\left\{(y_1,\dots,y_C)\in \prod_{l=1}^C B_{j_l}
:\max_{1\le l\le C}\Lambda_{j_l}(y_l)\ge \frac12\right\}.
\]
Consider the energy function
\[
\tilde{E}:S\times K\to \mathbb{R}
\]
defined by
\[
\tilde{E}\big((y_1,\dots,y_C),y\big)
=
\frac{\sum_{l=1}^C \Lambda_{j_l}(y_l)d(y_l,y)^2}
{\sum_{l=1}^C \Lambda_{j_l}(y_l)},
\qquad
(y_1,\dots,y_C)\in S,\quad y\in K.
\]
For each $(y_1,\dots,y_C)\in S$, let
\[
\tilde{h}(y_1,\dots,y_C)
\]
denote the center of mass associated with the energy function
$\tilde{E}\big((y_1,\dots,y_C),\cdot\big)$. Then, by construction,
\[
f_i(x)
=
\tilde{h}\big(h_{j_1}(x),h_{j_2}(x),\dots,h_{j_C}(x)\big),
\qquad
x\in B_{\frac1\epsilon}(x_0).
\]

Now consider the corresponding center of mass in the Euclidean space
$T_{y_0}K$, where it is given by the weighted linear average. Define
\[
\bar{h}:S\to T_{y_0}K
\]
by
\[
\bar{h}(y_1,\dots,y_C)
=
\frac{\sum_{l=1}^C
\Lambda_{j_l}(y_l)\exp_{y_0}^{-1}(y_l)}
{\sum_{l=1}^C\Lambda_{j_l}(y_l)}.
\]
Then $\exp_{y_0}^{-1}\circ\tilde{h}$
is $\Phi(\epsilon\mid N,k)$-$C^3$-close to $\bar{h}$. Indeed, under the
exponential coordinates centered at $y_0$, the metric on
$B_{\frac{10}{\epsilon}}(y_0)$ is $\epsilon^2$-$C^4$-close to the flat
metric on the corresponding ball in $T_{y_0}K$. Since the center of mass
with respect to the flat metric on $\mathbb{R}^k$ is precisely the
weighted linear average, the desired $C^3$-closeness follows.

For any $1 \le l \le C$, $\mathrm{exp}_{p^{j_l}}^{-1} \circ h_{j_l}$ is a harmonic $\Phi(\epsilon \,|\, N,k)$-GHA, and $\mathrm{exp}_{y_0}^{-1} \circ \mathrm{exp}_{p^{j_l}}$ is $\Phi(\epsilon \,|\, N,k)$-$C^3$-close to a linear isometry on $B_{\frac{3}{\epsilon}}(0^k)$ by our assumption. Thus
$$
\mathrm{exp}_{y_0}^{-1} \circ h_{j_l}
=
\mathrm{exp}_{y_0}^{-1} \circ \mathrm{exp}_{p^{j_l}} \circ \mathrm{exp}_{p^{j_l}}^{-1} \circ h_{j_l}
$$
is a $\Phi(\epsilon \,|\, N,k)$-GHA with
$$
\left| \Delta (\mathrm{exp}_{y_0}^{-1} \circ h_{j_l}) \right| \le \Phi(\epsilon \,|\, N,k), 
\quad 
\left| \nabla \Delta (\mathrm{exp}_{y_0}^{-1} \circ h_{j_l}) \right| \le \Phi(\epsilon \,|\, N,k)
$$
on $B_{\frac{1}{\epsilon}}(x_0)$.

Since $|\nabla^l \Lambda_j|\le \Phi(\epsilon \,|\, N,k)$ for $l=1,2,3$, the linear average
$$
\bar{h} \circ (h_{j_1},h_{j_2},\dots,h_{j_C})
=
\frac{\sum_{l=1}^C \Lambda_{j_l}(h_{j_l}(x)) \mathrm{exp}_{y_0}^{-1}(h_{j_l}(x))}
{\sum_{l=1}^C \Lambda_{j_l}(h_{j_l}(x))}
:
B_{\frac{1}{\epsilon}}(x_0) \to T_{y_0}K
$$
is a $\Phi(\epsilon \,|\, N,k)$-GHA with
$$
\left| \Delta \big(\bar{h}(h_{j_1},\dots,h_{j_C})\big) \right| \le \Phi(\epsilon \,|\, N,k), 
\quad 
\left| \nabla \Delta \big(\bar{h}(h_{j_1},\dots,h_{j_C})\big) \right| \le \Phi(\epsilon \,|\, N,k)
$$
on $B_{\frac{1}{\epsilon}}(x_0)$. In particular, $\bar{h}(h_{j_1},\dots,h_{j_C})$ is a $(k,\Phi(\epsilon \,|\, N,k))$-splitting map on $B_{\frac{1}{\epsilon}}(x_0)$.

Since $\mathrm{exp}_{y_0}^{-1} \circ \tilde{h}$ is $\Phi(\epsilon \,|\, N,k)$-$C^3$-close to $\bar{h}$, it follows that
$$
\mathrm{exp}_{y_0}^{-1} \circ f_i
=
\mathrm{exp}_{y_0}^{-1} \circ \tilde{h}(h_{j_1},\dots,h_{j_C})
$$
satisfies the condition in Lemma~\ref{Lemma: local splitting}, therefore is also a $(k,\Phi(\epsilon \,|\, N,k))$-splitting map on $B_{\frac{1}{\epsilon}}(x_0)$.

\end{proof}

\begin{remark}
We can also use the embedding argument in \cite{HondaPeng2024} to prove Theorem \ref{fibration}. 
By the Nash embedding theorem, we can isometrically embed $K$ into some $\mathbb{R}^n$, where $n$ depends only on $k$. Let the embedding map be $\Phi=(\phi_1,\phi_2,\dots,\phi_n):K \hookrightarrow \mathbb{R}^n$. Let $\pi$ be the projection map from a neighborhood of $K \subset \mathbb{R}^n$ onto $K$. By \cite{HondaPeng2024}, if $i$ is sufficiently large, we can construct $\Psi^i=(\psi_1^i,\psi_2^i,\dots,\psi_n^i): X_i \to \mathbb{R}^n$ such that 
$$
f_i:=\Phi^{-1} \circ \pi \circ \Psi^i
$$
is a GHA, and $\Delta \psi_j^i$ is $H^{1,2}$-close to $\Delta \phi_j$ for any $1 \le j \le n$. 
Then $f_i$ is locally almost $k$-splitting after rescaling the metric by a sufficiently large factor.
\end{remark}

\medskip
We next state a $G$-stability result in the RCD setting. 

\begin{theorem}[$G$-equivariant stability] \label{inv}
There exists $\epsilon(N)>0$ such that for any $0<\epsilon \le \epsilon(N)$, the following holds. Let $(X,d,\mathfrak{m})$ be a compact $\mathrm{RCD}(-(N-1)\epsilon^2,N)$ space and let $K$ be a $k$-manifold satisfying $\mathrm{inj}_K \ge \frac{10}{\epsilon}$, and such that for any $p \in K$, $B_{\frac{10}{\epsilon}}(p)$ is $\epsilon^2$-$C^4$-close, via $\mathrm{exp}_p^{-1}$, to its preimage in the tangent space $T_pK$ endowed with the flat metric. 

Assume that a map $h:X \to K$ is an $\epsilon$-GHA on any $\frac{10}{\epsilon}$-ball in $X$, that is, for any $x \in X$, the restriction
$$
h:B_{\frac{10}{\epsilon}}(x) \to K
$$
is an $\epsilon$-GHA onto its image. Suppose that a finite group $G$ acts isometrically on $X$ and on $K$, and that $h$ is $\epsilon$-almost $G$-equivariant. 

Then there exists a $G$-equivariant map $f_G:X \to K$ which is also a $\Phi(\epsilon \,|\, N,k)$-GHA and locally $(k,\Phi(\epsilon \,|\, N,k))$-splitting on every $\frac{1}{\epsilon}$-ball.
\end{theorem}

\begin{remark}
In Theorem \ref{inv}, there is no control on the diameter of $X$ or on the order $|G|$. Instead, we only require a global map $h$, and the gluing procedure is local. If we further assume $k=N$ and $\epsilon$ is sufficiently small, then $f_G$ is bi-H\"{o}lder on any $\frac{1}{\epsilon}$-ball by Theorem \ref{Rei}.
\end{remark}

\begin{proof}
We use the same argument as in the proof of Theorem \ref{local split} to construct a map $f:X \to K$ which is $(k,\Phi(\epsilon \,|\, N,k))$-splitting and a $\Phi(\epsilon \,|\, N,k)$-GHA on any $\frac{1}{\epsilon}$-ball. Notice that the assumption that $h:X \to K$ is an $\epsilon$-GHA on any $\frac{10}{\epsilon}$-ball is sufficient for the construction in Theorem \ref{local split}, since the gluing and averaging procedures are local in nature.

Since $f(x)$ is close to $h(x)$ for any $x \in X$, we have 
$$
d(f(gx),g f(x)) \le \Phi(\epsilon \,|\, N,k)
$$
for any $x \in X$ and $g \in G$. In particular, for any $g \in G$, the map
$$
x \mapsto g f(g^{-1}x) : X \to K
$$
is also a $\Phi(\epsilon \,|\, N,k)$-GHA.

Since $f$ is $(k,\Phi(\epsilon \,|\, N,k))$-splitting on any $\frac{1}{\epsilon}$-ball and the $G$-actions are isometric on $X$ and $K$, it follows that for any $g \in G$, the map
$$
x \mapsto g f(g^{-1}x)
$$
is also $(k,\Phi(\epsilon \,|\, N,k))$-splitting on any $\frac{1}{\epsilon}$-ball. 

Now we average over the $G$-action by taking the center of mass. Define the energy function 
$$
E(x,y): X \times K \to \mathbb{R}, 
\quad 
E(x,y)= \frac{\sum_{g \in G} d(g f(g^{-1}x),y)^2}{|G|}.
$$
For any fixed $x \in X$, $E(x,\cdot)$ is strictly convex in $B_{\frac{1}{\epsilon}}(f(x))$, hence it attains a unique minimum point $y$. Define $f_G(x)=y$. Then $f_G$ is $G$-equivariant by uniqueness of the minimizer. Moreover, $f_G$ is a $\Phi(\epsilon \,|\, N,k)$-GHA.

We next show that $f_G$ is $(k,\Phi(\epsilon \,|\, N,k))$-splitting on every $\frac{1}{\epsilon}$-ball by an argument similar to that in the proof of Theorem \ref{local split}. Fix $x_0 \in X$ and let $y_0=f_G(x_0)$. Define the energy function on the product space $\tilde{E}: \prod_{l=1}^{|G|} B_{\frac{1}{\epsilon}}(y_0) \times K \to \mathbb{R}$ by
$$
\tilde{E}\Big(\prod_{l=1}^{|G|} y_l,y\Big) 
= \frac{\sum_{l=1}^{|G|} d(y_l,y)^2}{|G|}, 
\quad 
y \in K,\ y_l \in B_{\frac{1}{\epsilon}}(y_0).
$$
Let $\tilde{h}\big(\prod_{l=1}^{|G|} y_l\big)$ be the center of mass with respect to $\tilde{E}$. Then by definition,
$$
f_G(x)=\tilde{h}\big(\prod_{g \in G} g f(g^{-1}x)\big).
$$

Now consider the center of mass in Euclidean space, which is the linear average. Define 
$$
\bar{h}: \prod_{l=1}^{|G|} B_{\frac{1}{\epsilon}}(y_0) \to T_{y_0}K
$$
by
$$
\bar{h}\Big(\prod_{l=1}^{|G|} y_l\Big) 
= \frac{\sum_{l=1}^{|G|} \mathrm{exp}_{y_0}^{-1}(y_l)}{|G|}.
$$
Then $\mathrm{exp}_{y_0}^{-1} \circ \tilde{h}$ is $\Phi(\epsilon \,|\, N,k)$-$C^3$-close to $\bar{h}$, since the metric on $B_{\frac{10}{\epsilon}}(y_0)$ is $\epsilon^2$-$C^4$-close to a $B_{\frac{10}{\epsilon}}$-ball in $T_{y_0}K$ with the flat metric.

The linear average
$$
\bar{h}\big(\prod_{g \in G} g f(g^{-1}x)\big) 
: B_{\frac{1}{\epsilon}}(x_0) \to T_{y_0}K
$$
is $(k,\Phi(\epsilon \,|\, N,k))$-splitting by the same argument as in Theorem \ref{local split}. Therefore, applying Lemma~\ref{Lemma: local splitting},
$$
\mathrm{exp}_{y_0}^{-1} \circ f_G
=
\mathrm{exp}_{y_0}^{-1} \circ \tilde{h}\big(\prod_{g \in G} g f(g^{-1}x)\big)
$$
is also $(k,\Phi(\epsilon \,|\, N,k))$-splitting on $B_{\frac{1}{\epsilon}}(x_0)$, as desired.
\end{proof}

\section{Proof of Theorem A: Construction of an Infranilmanifold} \label{proof A}

In this section, we prove Theorem A.

\begin{thm1}
Given $N,v>0$, there exist constants $\epsilon(N,v) > 0$ and $C(N)>0$ such that if a $\mathrm{RCD}(-(N-1),N)$ space $(X,d,\mathcal{H}^N)$ satisfies
$$
\mathrm{diam}(X) \le \epsilon \le \epsilon(N,v), \quad
\mathcal{H}^N(B_1(\tilde{p})) \ge v, \quad \forall \tilde{p} \in \widetilde{X},
$$
then $X$ is bi-H\"older homeomorphic to an infranilmanifold $\mathcal{N}/\Gamma$, where $\mathcal{N}$ is a simply connected $N$-dimensional nilpotent Lie group equipped with a left-invariant metric, $\Gamma$ is a discrete subgroup of $\mathcal{N} \rtimes \mathrm{Aut}(\mathcal{N})$, and
$[ \Gamma : \Gamma \cap \mathcal{N}] \le C(N)$.

More precisely, there exists a homeomorphism $f:X \to \mathcal{N}/\Gamma$ such that, for any $x,y \in X$,
$$
(1-\Phi(\epsilon \,|\, N,v))d(x,y)^{1+\Phi(\epsilon \,|\, N,v)} 
\le d(f(x),f(y)) 
\le (1+\Phi(\epsilon \,|\, N,v))d(x,y),
$$
where $\Phi(\epsilon \,|\, N,v) \to 0$ as $\epsilon \to 0$.

Moreover, if $X$ is a smooth $N$-manifold with $\mathrm{Ric} \ge -(N-1)$, then $f$ is a diffeomorphism.
\end{thm1}

Suppose the theorem is false. Then there exists a sequence of $\mathrm{RCD}(-(N-1),N)$ spaces $(X_i,d_i,\mathcal{H}^N)$ with non-collapsing universal covers, diameters $\mathrm{diam}(X_i)\to 0$, and such that none of the $X_i$ are bi-Hölder homeomorphic to any infranilmanifold.

After rescaling by suitable factors $r_i \to \infty$ and passing to a subsequence, we may assume that each $(X_i,d_i,\mathcal{H}^N)$ are $\mathrm{RCD}(-\epsilon_i^2(N-1),N)$ spaces, with $\epsilon_i \to 0$, and that the universal covers $\widetilde{X}_i$ converge in the pointed Gromov--Hausdorff sense to $\mathbb{R}^N$.

By Theorem~\ref{KW}, for sufficiently large $i$, the fundamental group
\[
G_i'=\pi_1(X_i,p_i)
\]
contains a normal nilpotent subgroup $G_i$ of index $\le C$ and nilpotent steps $\le N$. Let $X_i'=\widetilde{X}_i/G_i$. Then

\[
\begin{CD}
(\widetilde{X}_i,\tilde{p}_i,G_i,G_i') @>eGH>> (\mathbb{R}^N,\tilde{p},G,G')\\
@VV\pi V @VV\pi V\\
(X_i',p_i',G_i'/G_i) @>GH>> \mathrm{pt}\\
@VV\pi V @VV\pi V\\
(X_i,p_i) @>GH>> \mathrm{pt}.
\end{CD}
\]

Since $X_i'$ is a finite cover of uniformly bounded degree, $\mathrm{diam}(X_i')\to0$.

We prove that the sequence $\{G_i\}$ has no small subgroups.

\begin{lemma}\label{no small subgroup}
There exists $\delta>0$ such that
\[
G_i(\delta)
=
\{ g \in G_i' \mid d(x,gx)<\delta, \quad \forall x \in B_{1/\delta}(\tilde{p}_i) \}
\]
contains no nontrivial subgroup for sufficiently large $i$.
\end{lemma}

\begin{proof}
Suppose by contradiction that there exist nontrivial small subgroups $H_i \subset G_i(\epsilon_i)$ with $\epsilon_i \to 0$. Then $H_i$ converges to the trivial group, and $\widetilde{X}_i/H_i$ converges to $\mathbb{R}^N$.

Let $\pi:\widetilde{X}_i \to \widetilde{X}_i/H_i$ be the projection map. By volume convergence, the unit ball in $\widetilde{X}_i/H_i$ centered at $\pi(\tilde{p}_i)$ has volume arbitrarily close to that of the unit ball $B_1(\tilde{p}_i) \subset \widetilde{X}_i$. 

However, any nontrivial deck transformation identifies distinct points in the ball, and therefore $\mathrm{Vol}\big(B_1(\tilde{p}_i)\big) \approx |H_i| \mathrm{Vol}\big(B_1(\pi(\tilde{p}_i))\big)$.  This yields a contradiction if $|H_i| \ge 2$.
\end{proof}

Then we can show that the $G$-action is free. 
\begin{lemma}\label{free}
The limit group $G$ acts freely and transitively on $\mathbb{R}^N$. In particular, $G\simeq \mathbb{R}^N$.
\end{lemma}

\begin{proof}
The $G$-action is transitive since the quotient space is a point. Let
\[
G_x=\{g \in G: gx=x\}.
\]
Then $G_x$ is compact. Since $G$ is nilpotent, the commutator subgroup satisfies $[G_x,G_0]=\{e\}$, where $G_0$ denotes the identity component. As $G$ acts transitively on $\mathbb{R}^N$, its identity component $G_0$ also acts transitively. Hence $G_x$ fixes every point in $\mathbb{R}^N$. Since $G$ acts effectively on $\mathbb{R}^N$, the stabilizer $G_x$ is trivial, so the action is free.
\end{proof}

\begin{corollary}\label{nofix}
For any $\delta>0$ there exists $\epsilon>0$ such that
\[
G_i(\tilde{p}_i,\epsilon) \subset G_i(\delta)
\]
for sufficiently large $i$.
\end{corollary}

\medskip

Using the structure of approximate groups, we can prove that $G_i$ is isomorphic to a lattice in a simply connected nilpotent Lie group.

\begin{lemma}\label{ni}
For sufficiently large $i$, $G_i$ is isomorphic to a lattice in an $N$-dimensional simply connected nilpotent Lie group $\mathcal{N}_i$.
\end{lemma}

\begin{proof}
Since the sequence $\{G_i\}$ has no small subgroups and converges to $\mathbb{R}^N$, Theorem~\ref{nil} produces a nilprogression $P_i$ of dimension $N$ containing $G_i(\delta)$. By Corollary~\ref{nofix}, it also contains $G_i(\tilde{p}_i,\epsilon)$.

For $i$ large with $\mathrm{diam}(X_i)<\epsilon/20$, Theorem~\ref{local} implies that the groupification of $P_i$ equals $G_i$. When $\mathrm{thick}(P_i)$ is sufficiently large, Theorem~\ref{Malcev} shows that this groupification is isomorphic to a lattice in a simply connected nilpotent Lie group $\mathcal{N}_i$.
\end{proof}

We next construct a left-invariant metric $g_{\mathcal{N}_i}$ on $\mathcal{N}_i$ which is locally $C^4$-close to the flat metric on $\mathbb{R}^N$. This metric was introduced by Zamora--Zhu in \cite{ZamoraZhu2024}. We include the construction in detail for the reader's convenience.

We use the notation introduced before Theorem~\ref{Q}. In particular, $\{v_{i,j}\}$ is a strong Malcev basis of the Lie algebra $\mathfrak{n}_i$ of $\mathcal{N}_i$, and $\{v_{i,j}\}$ converges to $\{v_j\}$, which is a strong Malcev basis of $\mathbb{R}^N$. Define
\[
g_{\mathcal{N}_i}(v_{i,j_1},v_{i,j_2})
=
\langle v_{j_1},v_{j_2}\rangle_{\mathbb{R}^N}.
\]

\begin{lemma}\label{metric}
There exists $\epsilon_i\to 0$ such that $B_{1/\epsilon_i}(e) \subset \mathcal{N}_i$ is $\epsilon_i$-$C^4$-close to $B_{1/\epsilon_i}(0^N) \subset \mathbb{R}^N$. In particular, $\mathrm{inj}_{\mathcal{N}_i}\ge\frac{1}{2\epsilon_i}$.
\end{lemma}

\begin{proof}
Let $\epsilon_i \to 0$ denote a sequence of positive numbers converging to $0$; their values may change from line to line.

By Theorem~\ref{Q}, the structure constants of $\mathfrak{n}_i$ with respect to the Malcev basis converge to zero. Hence the Lie bracket satisfies
\[
||[v_{i,j_1}, v_{i,j_2}]|| \le \epsilon_i.
\]

By the Koszul formula for left-invariant metrics, the Levi-Civita connection coefficients are linear combinations of the structure constants. Therefore the Christoffel symbols are bounded by $\epsilon_i$.

Since the metric is left-invariant, the curvature tensor and its covariant derivatives are algebraic expressions in the structure constants. Hence the curvature tensor and its covariant derivatives up to third order are bounded by $\epsilon_i$.

Hence, on balls of radius \(1/\epsilon_i\), the metric is \(\epsilon_i\)-\(C^4\)-close to the flat metric on \(\mathbb{R}^N\). The injectivity radius estimate follows from standard comparison results for almost flat manifolds.
\end{proof}
\begin{remark}
In fact, one can obtain $C^{\infty}$-closeness in Lemma~\ref{metric}, although $C^4$-closeness is sufficient for our purposes.
\end{remark}

From the proof of Lemma~\ref{metric}, the map 
\[
\phi_i : (\mathcal{N}_i, e) \longrightarrow (\mathbb{R}^N, 0^N)
\]
defined as in Theorem~\ref{str}~(2) is an \(\epsilon_i\)-GHA.

\begin{lemma}\label{localeg}
For each fixed $\epsilon>0$ and the identity element $e \in \mathcal{N}_i$, there exists
\[
h_i': B_{1/\epsilon}(\tilde{p}_i)\to B_{1/\epsilon}(e) \subset \mathcal{N}_i
\]
which is an $\epsilon_i$-GHA and $\epsilon_i$-almost $G_i(\tilde{p}_i,1/\epsilon)$-equivariant, with $\epsilon_i\to0$.
\end{lemma}

\begin{proof}
Consider two GHAs  $$\psi_i : (\widetilde{X}_i,\tilde{p}_i) \to (\mathbb{R}^N,0^N), \, \phi_i : (\mathcal{N}_i, e) \to (\mathbb{R}^N, 0^N).$$ 
Then
$$h_i':= \phi_i^{-1} \circ \psi_i: (\widetilde{X}_i,\tilde{p}_i) \to (\mathcal{N}_i, e)$$ is also a GHA.

The $\epsilon_i$-almost equivariance follows from the construction of the nilprogression approximating $G_i$ and from the compatibility of the left-invariant metric $g_{\mathcal{N}_i}$ with the Malcev coordinates. Indeed, the actions of $g_i \in G_i(\tilde{p}_i,1/\epsilon)$ on $\widetilde{X}_i$ and $\mathcal{N}_i$ are eGH-close to the same action on $\mathbb{R}^N$, so $h_i'$ is almost equivariant.
\end{proof}

Since diameters of $\widetilde{X}_i/G_i$ and $\mathcal{N}_i/G_i$ are small, we may equivariantly extend the above local map $h_i'$ to a global map.

\begin{lemma}[{\cite{Wang2023}}]\label{extension}
There exists a global map $h_i:\widetilde{X}_i\to\mathcal{N}_i$ which is $\epsilon_i$-almost $G_i$-equivariant and is an $\epsilon_i$-GHA on any ball of radius $\frac{1}{\epsilon_i}$.
\end{lemma}

\medskip
So far we have only analyzed the actions of $G_i$. We now turn to the actions of $G_i'$. Fix a sufficiently large $R>0$, and let $H_i \le G_i'$ be the subgroup generated by elements in $G_i'(\tilde{p}_i,R)$ which are close to $G=\mathbb{R}^N$. By \cite[Theorem 3.10]{FukayaYamaguchi1992}, it follows that $H_i$ is a normal subgroup, $H_i$ converges to $\mathbb{R}^N$, and $G'/G \cong G_i'/H_i$.

Since $G_i$ is boundedly generated and converges to $\mathbb{R}^N$, we have $G_i \le H_i$. Applying the argument of Lemma~\ref{ni} to $H_i$, we conclude that $H_i$ is a lattice in a simply connected nilpotent Lie group. Therefore, after replacing $G_i$ if necessary, we may assume $G_i = H_i$ and $G_i'/G_i \cong G'/G$.

\begin{lemma}\label{G_i'-inv}
For all sufficiently large $i$, the group $G_i'$ can be identified with a discrete subgroup of $\mathcal{N}_i \rtimes \mathrm{Aut}(\mathcal{N}_i)$. After a slight modification of the metric on $\mathcal{N}_i$, we may assume that $G_i'$ acts isometrically on $\mathcal{N}_i$. Moreover, the map $h_i : \widetilde{X}_i \to \mathcal{N}_i$ from Lemma~\ref{extension} is $\epsilon_i$-almost $G_i'$-equivariant.
\end{lemma}

\begin{proof}
We construct an injective homomorphism 
\[
\phi_i : G_i' \to \mathcal{N}_i \rtimes \mathrm{Aut}(\mathcal{N}_i).
\]
Fix $R>0$ sufficiently large. Then there exist elements $\{g_j\}_{j=1}^J \subset G'(\tilde{p},R)$, $g_1=e$, such that
\[
G' = \bigsqcup_{j=1}^J G g_j.
\]

Since the limit space is $\mathbb{R}^N$, we may write $g_j = (b_j,A_j)$, where $b_j \in \mathbb{R}^N$ and $A_j \in O(N)$. 
For each $j$, choose $g_{ij} \in G_i'$ and $b_{ij} \in \mathcal{N}_i$ converging to $g_j \in G'$ and $b_j \in \mathbb{R}^N$, respectively. Since $G_i'/G_i \cong G'/G$, we have
\[
G_i' = \bigsqcup_{j=1}^J G_i g_{ij}.
\]

By rigidity, conjugation by $g_{ij}$ on the lattice $G_i$ induces an automorphism of $\mathcal{N}_i$, denoted
$C_{g_{ij}} \in \mathrm{Aut}(\mathcal{N}_i)$.

Moreover, $C_{g_{ij}}$ converges in the equivariant Gromov--Hausdorff sense to $C_{g_j} = A_j$. In particular, for each $j$ and all sufficiently large $i$, $C_{g_{ij}}$ is almost isometric.

Since $g_1=e \in G'$, we may choose that $g_{i1}$, $b_{i1}$ and $C_{g_{i1}}$ are all trivial. 
Moreover, for any $j \neq 1$, $A_j$ is bounded away from the identity. That is, there exists $\delta>0$ such that $d(g_j,e) \ge \delta$. it follows that for all sufficiently large $i$, the corresponding automorphisms $C_{g_{ij}}$ are at least $\delta/2$ away from the identity.

For any $h \in G_i$ and each $j$, define a map $\phi_i:G_i' \to \mathcal{N}_i \rtimes \mathrm{Aut}(\mathcal{N}_i)$ by
\[
\phi_i(h g_{ij}) := (h b_{ij}, C_{g_j}).
\]
Then  $\phi_i$ is a homomorphism on $G_i$ and an almost homomorphism on $G_i'$; that is, for any $1 \le j_1,j_2 \le J$, $\phi_i(g_{ij_1})\,\phi_i(g_{ij_2})$
is equivariantly Gromov--Hausdorff close to $\phi_i(g_{ij_1}g_{ij_2})$. The failure of exact multiplicativity comes from the finite quotient $G_i'/G_i$. Using a center-of-mass argument, we may slightly modify $\phi_i$ to obtain an actual homomorphism.

Moreover, $\phi_i$ is injective since for any $j \neq 1$, $C_{g_{ij}}$ is at least $\delta/3$ away from the identity.

The action of $G_i$ on $\mathcal{N}_i$ is isometric. Since the action of $G_i'$ on $\mathcal{N}_i$ is almost isometric and the quotient $G_i'/G_i$ is finite, we can average the left-invariant metric on $\mathcal{N}_i$ over $G_i'/G_i$. After this modification, we may assume that the $G_i'$-action on $\mathcal{N}_i$ is isometric.

Recall that the map
\[
h_i' : B_{1/\epsilon}(\tilde{p}_i) \to B_{1/\epsilon}(e) \subset \mathcal{N}_i
\]
from Lemma~\ref{localeg} is an $\epsilon_i$-almost $G_i$-equivariant Gromov--Hausdorff approximation. Since both the $G_i'$-actions on $\widetilde{X}_i$ and $\mathcal{N}_i$ converge to the same $G'$-action on $\mathbb{R}^N$, it follows that $h_i'$ is also $\epsilon_i$-almost $G_i'$-equivariant. Therefore, the gluing map
\[
h_i : \widetilde{X}_i \to \mathcal{N}_i
\]
from Lemma~\ref{extension} is likewise $\epsilon_i$-almost $G_i'$-equivariant.
\end{proof}

\begin{remark}
The first statement in Lemma~\ref{G_i'-inv} can also be obtained as follows. By \cite{ZamoraZhu2024}, $\widetilde{X}_i$ is aspherical, and hence $G_i'$ is torsion-free. Together with the fact that $G_i$ is a lattice in $\mathcal{N}_i$ and $G_i'/G_i$ is finite, we may apply \cite{LeeRay} to conclude that $G_i'$ is isomorphic to a discrete subgroup of $\mathcal{N}_i \rtimes \mathrm{Aut}(\mathcal{N}_i)$.
\end{remark}

The finitely generated nilpotent group $G_i$ is residually finite, i.e., for any $g \in G_i \setminus \{e\}$, there exists a finite-index normal subgroup of $G_i$ which does not contain $g$. Then there is a normal subgroup $G_i''$ of $G_i$ with finite index so that $$G_i'' \cap B_{\frac{1}{\epsilon}}(e) = \{e\},$$
where $B_{\frac{1}{\epsilon}}(e)$ denotes the ball in $\mathcal{N}_i$.
 By Lemma~\ref{normal}, we may further assume that $G_i''$ is normal in $G_i'$.

\medskip

Now we can finish the proof of Theorem A.
\begin{proof}[Proof of Theorem A]
Taking the quotient by $G_i''$, the almost equivariant map in Lemma~\ref{G_i'-inv} descends to
\[
\bar{h}_i:
(\widetilde{X}_i/G_i'',G_i'/G_i'')
\longrightarrow
(\mathcal{N}_i/G_i'',G_i'/G_i'').
\]
By the choice of $G_i''$, the injectivity radius of $\mathcal{N}_i/G_i''$ is at least $1$.

Applying Theorem~\ref{inv} yields a $G_i'/G_i''$-equivariant locally almost splitting map 
$$f_i:\widetilde{X}_i/G_i'' \to \mathcal{N}_i/G_i''.$$ 
By \cite{HondaPeng2024}, $f_i$ is a bi-H\"older homeomorphism. We obtain the bi-H\"older estimate

\[
(1-\Phi(\epsilon \,|\, N))d_i(x,y)^{1+\Phi(\epsilon \,|\, N)}
\le
d(f_i(x),f_i(y))
\le
(1+\Phi(\epsilon \,|\, N))d_i(x,y).
\]

$f_i$ is $G_i'/G_i''$-equivariant. Descending to $\bar{f}_i: X_i=\widetilde{X}_i/G_i' \to \mathcal{N}_i/G_i'$, we conclude that $X_i$ is bi-H\"older homeomorphic to an infranilmanifold, which contradicts our initial assumption.

If $X_i$ is smooth with $\mathrm{Ric}\ge-(N-1)$, $\bar{f}_i$ is smooth and $d\bar{f}_i$ is nondegenerate. Therefore $\bar{f}_i$ is a diffeomorphism.
\end{proof}
\begin{remark} At the beginning of the proof of Theorem A, we rescale the metrics by factors $r_i \to \infty$. Therefore, for a contradiction sequence $(X_i,d_i,\mathcal{H}^N)$, we actually proved the bi-H\"{o}lder estimate for the rescaled spaces $(X_i,r_id_i,\mathcal{H}^N)$, where $r_i \to \infty$. More precisely, $$(1-\Phi(\epsilon \,|\, N))(r_id_i(x,y))^{1+\Phi(\epsilon \,|\, N)} \le d_{\mathcal{N}_i}(f(x),f(y)) \le (1+\Phi(\epsilon \,|\, N))r_id_i(x,y),$$ where $d_{\mathcal{N}_i}$ is the distance induced by $g_{\mathcal{N}_i}$ on $\mathcal{N}_i$. 

Fix a large $i$, define a rescaled metric $$d_i':= \frac{d_{\mathcal{N}_i}}{r_i}$$ on $\mathcal{N}_i$. Since $r_i \ge 1$, we obtain $$(1-\Phi(\epsilon \,|\, N))(d_i(x,y))^{1+\Phi(\epsilon \,|\, N)} \le d_i'(f(x),f(y)) \le (1+\Phi(\epsilon \,|\, N))d_i(x,y).$$ Then $f:(X_i,d_i) \to (\mathcal{N}_i,d_i')$ is also bi-H\"{o}lder. Thus the bi-H\"{o}lder estimate on the blown-up metric implies the corresponding estimate on the original metric. 

For this reason, in the subsequent sections (e.g., the proofs of Lemma~\ref{product} and Theorem~B), we suppress explicit reference to the rescaling procedure and directly apply Theorem~\ref{inv} together with the canonical Reifenberg method.
\end{remark}



\section{Mixed curvature and the almost flat manifold theorem}

In this section, we prove Theorem~\ref{mixed} using a construction similar to the one in the proof of Theorem~A.
The key difference is that we glue strainer maps rather than almost splitting maps. We begin by stating a version of Theorem~\ref{inv} in the mixed curvature setting.

Fix $\epsilon>0$ sufficiently small. 
Assume that $(X,d,\mathfrak m,p)$ is an $\mathrm{RCD}(-\epsilon,N)$ space such that $(X,d)$ is $\mathrm{CBA}(\epsilon)$, $\partial X=\emptyset$ and the rectifiable dimension of $X$ is $k$. 
Let $(K,q)$ be a smooth Riemannian $k$-manifold with $\mathrm{inj}_K \ge \frac{10}{\epsilon}$. Assume that the ball $B_{\frac{10}{\epsilon}}(q)$ is $\epsilon^2$-$C^4$-close, via $\exp_q^{-1}$, to its preimage in the tangent space $T_qK$ equipped with the flat metric.
Suppose that $h:B_2(p)\to B_2(q)$
is an $\epsilon$-GHA.

By \cite{Kapovitch2021}, one can construct, using a strainer, a differentiable $\Phi(\epsilon \,|\, N,k)$-GHA
\[
u:B_1(p)\to B_2(q)
\]
satisfying
\[
1-\Phi(\epsilon \,|\, N,k)\le |\nabla u|\le 1+\Phi(\epsilon \,|\, N,k).
\]
In particular, $u$ is a diffeomorphism onto its image.
Moreover, if another differentiable $\Phi(\epsilon \,|\, N,k)$-GHA
\[
v:B_1(p)\to B_2(q)
\]
is constructed using another strainer, then $v$ is $\Phi(\epsilon \,|\, N,k)$-$C^1$-close to $u$.

Using the same gluing construction as in the proof of Theorem~\ref{inv}, we can glue local strainer maps to obtain the following result. 
See also the proof of the fibration theorem for mixed curvature spaces in \cite{Kapovitch2021}.

\begin{theorem}\label{mix glue}
Given $N$ and $k \le N$, there exists $\epsilon(N,k)>0$ such that for any $\epsilon\le\epsilon(N,k)$ the following holds.

Let $K$ be a smooth $k$-manifold satisfying $\mathrm{inj}_K \ge \frac{10}{\epsilon}$, and such that for every $p\in K$, the ball $B_{\frac{10}{\epsilon}}(p)$ is $\epsilon^2$-$C^4$-close, via $\mathrm{exp}_p^{-1}$, to its preimage in $T_pK$ equipped with the flat metric.

Let $(X,d,\mathfrak m)$ be an $\mathrm{RCD}(-\epsilon,N)$ space such that $(X,d)$ is $\mathrm{CBA}(\epsilon)$, $\partial X=\emptyset$, and $\dim(X)=k$.

Assume that there exists a global map $h:X\to K$ which is an $\epsilon$-Gromov--Hausdorff approximation on every ball of radius $\frac{10}{\epsilon}$ in $X$. 
Suppose moreover that a finite group $G$ acts isometrically on both $X$ and $K$, and that $h$ is $\epsilon$-almost $G$-equivariant.

Then there exists a $G$-equivariant map 
\[
f_G:X\to K
\]
which is a $\Phi(\epsilon \,|\, N,k)$-GHA and is bi-Lipschitz and a diffeomorphism on every ball of radius $\frac{1}{\epsilon}$. 
More precisely, for any $x,y\in X$ with $d(x,y)<\frac{1}{\epsilon}$,
\[
(1-\Phi(\epsilon \,|\, N,k))\,d(x,y)
\le d\bigl(f_G(x),f_G(y)\bigr)
\le (1+\Phi(\epsilon \,|\, N,k))\,d(x,y),
\]
where $\Phi(\epsilon \,|\, N,k)\to 0$ as $\epsilon\to 0$.
\end{theorem}

\medskip

We first note that Lemma~\ref{no small subgroup} remains valid if the volume bound assumption is replaced by the condition that the rectifiable dimensions do not drop.

\begin{lemma}[ {\cite[Theorem 94]{SantosZamora}} ] \label{no small}
Let $(X_i,p_i,G_i,\mathfrak{m}_i)$ be a sequence of $\mathrm{RCD}(K,N)$ spaces. Assume that the following equivariant measured Gromov--Hausdorff convergence holds:
\[
(X_i,p_i,G_i,\mathfrak{m}_i) \xrightarrow{\mathrm{emGH}} (X,p,G,\mathfrak{m}).
\]
Assume that for all sufficiently large $i$, the rectifiable dimensions satisfy
\[
\dim(X_i)=\dim(X)=k.
\]
Then the sequence $\{G_i\}$ has no small subgroups.
\end{lemma}

Note that if $k<N$, the volume convergence theorem is not applicable.






We now proceed to prove Theorem \ref{mixed} using Theorem \ref{mix glue} and the construction in the proof of Theorem~A.

\begin{proof}[Proof of Theorem \ref{mixed}]
We argue by contradiction. Assume that there exists a sequence $\epsilon_i \to 0$ and spaces $(X_i,d_i,\mathfrak{m}_i)$ such that each $(X_i,d_i,\mathfrak{m}_i)$ is $\mathrm{RCD}(-\epsilon_i,N)$ and $(X_i,d_i)$ is $\mathrm{CBA}(\epsilon_i)$, with $\partial X_i=\emptyset$ and $\mathrm{diam}(X_i)\le \epsilon_i$. Suppose moreover that none of the $X_i$ is bi-Lipschitz diffeomorphic to an infranilmanifold of dimension $\le N$. We will derive a contradiction by showing that $X_i$ is bi-Lipschitz diffeomorphic to an infranilmanifold of dimension $\le N$ for sufficiently large $i$.

Passing to a subsequence if necessary, we may assume that all $X_i$ are manifolds of the same dimension $k \le N$. We recall the argument from \cite{Kapovitch2021} that the universal cover $\widetilde{X}_i$ of $X_i$ converges in the pointed Gromov--Hausdorff sense to $\mathbb{R}^k$.

Fix $p_i\in X_i$ and a large $r>0$. Consider 
\[
\hat{X}_i=B_r(0^k)\subset T_{p_i}X_i,
\]
the $r$-ball centered at the origin in the tangent space $T_{p_i}X_i$, equipped with the pullback metric via $\mathrm{exp}_{p_i}$. Since the curvature of $X_i$ is bounded above by $\epsilon_i\to 0$, the exponential map 
\[
\mathrm{exp}_{p_i}:\hat{X}_i\to X_i
\]
is non-degenerate for sufficiently large $i$. Thus $\hat{X}_i$ is a pseudo-cover of $X_i$ with pseudo-actions. By Lemma~3.7 in \cite{Kapovitch2021}, $\hat{X}_i$ converges in the pGH sense to the $r$-ball in $\mathbb{R}^k$.

By \cite[Lemma~2.5]{Kapovitch2021}, the groupification of these pseudo-actions is $\pi_1(X_i,p_i)$. Therefore, the gluing space constructed from the pseudo-cover and pseudo-actions (see Section~2.3) is precisely the universal cover $\widetilde{X}_i$, which also converges in the pGH sense to $\mathbb{R}^k$.

Let $\tilde{p}_i$ be a lift of $p_i$ in $\widetilde{X}_i$. By the generalized Margulis lemma, for sufficiently large $i$, $G_i'=\pi_1(X_i,p_i)$ contains a nilpotent subgroup $G_i$ of index $\le C$. After rescaling the measure on $\widetilde{X}_i$, set $X_i'=\widetilde{X}_i/G_i$. By the splitting theorem, we obtain the following diagram:
\begin{center}
$\begin{CD}
(\widetilde{X}_i,\tilde{p}_i,G_i,G_i') @>eGH>> (\mathbb{R}^k,\tilde{p},G,G')\\
@VV\pi V @VV\pi V\\
(X_i',p_i',G_i'/G_i) @>GH>> \mathrm{pt}\\
@VV\pi V @VV\pi V\\
(X_i,p_i) @>GH>> \mathrm{pt}
\end{CD}$
\end{center}

By Lemma~\ref{free}, $G$ is free and transitive, hence can be identified with $\mathbb{R}^k$. By Lemma~\ref{no small}, the sequence $\{G_i\}$ has no small subgroups.

By Lemma~\ref{ni}, for sufficiently large $i$, $G_i$ is isomorphic to the lattice in a $k$-dimensional simply connected nilpotent Lie group $\mathcal{N}_i$. 

For any fixed $\epsilon>0$ and large $i$, applying Lemmas~\ref{G_i'-inv}, we can similarly endow $\mathcal{N}_i$ with a left-invariant and $G_i'$-invariant metric such that $\mathcal{N}_i$ is $\epsilon$-$C^4$-close to $\mathbb{R}^k$. Moreover, we obtain a map 
\[
h_i:\widetilde{X}_i \to \mathcal{N}_i,
\]
which is an $\epsilon_i$-GHA on every $\frac{1}{\epsilon}$-ball in $\widetilde{X}_i$ and is $\epsilon_i$-almost $G_i'$-equivariant. 

We can choose a normal subgroup $G_i''\unlhd G_i'$ of finite index such that $G_i'' \le G_i$ and 
\[
G_i'' \cap B_{\frac{1}{\epsilon}}(e)=\{e\},
\]
where $e\in \mathcal{N}_i$ is the identity element.  In particular, $\mathcal{N}_i/G_i''$ is compact and $\epsilon$-$C^4$-close to $\mathbb{R}^k$ on every $\frac{1}{\epsilon}$-ball. This induces a map
\[
\bar{h}_i: \widetilde{X}_i/G_i'' \to \mathcal{N}_i/G_i'',
\]
which is an $\epsilon_i$-GHA on every $\frac{1}{\epsilon}$-ball and is $\epsilon_i$-almost $G_i'/G_i''$-equivariant.

$\bar{h}_i: \widetilde{X}_i/G_i'' \to \mathcal{N}_i/G_i''$ satisfies the assumptions of Theorem \ref{mix glue}. Applying Theorem~\ref{mix glue}, we obtain a $G_i'/G_i''$-equivariant map
\[
f_i : \widetilde{X}_i/G_i'' \to \mathcal{N}_i/G_i'',
\]
which is bi-Lipschitz and diffeomorphic on every $\frac{1}{\epsilon}$-ball.

Since $f_i$ is $G_i'/G_i''$-equivariant, it descends to a bi-Lipschitz diffeomorphism on the quotient
\[
\bar{f}_i : X_i=\widetilde{X}_i/G_i' \longrightarrow \mathcal{N}_i/G_i'.
\]
Therefore, for sufficiently large $i$, $X_i$ is bi-Lipschitz diffeomorphic to the infranilmanifold $\mathcal{N}_i/G_i'$, contradicting the initial assumption that no $X_i$ is bi-Lipschitz diffeomorphic to an infranilmanifold of dimension $\le N$.
\end{proof}

\section{Proof of Theorem B: construct fibers along the collapsing direction}
We now restate Theorem~B.
\begin{thm2} 
Given $N,v,\rho>0$, suppose that a sequence of compact $\mathrm{RCD}(-(N-1),N)$ spaces $(X_i,d_i,\mathcal{H}^N)$ with $(\rho,v)$-bounded covering geometry converges in the Gromov--Hausdorff sense to a compact smooth Riemannian $k$-manifold $K$. 
Then, for sufficiently large $i$, there exists a fiber bundle map
$$
f_i:X_i \to K
$$
which is an $\epsilon_i$-GHA, with $\epsilon_i \to 0$ as $i \to \infty$. 
Moreover, the fibers are homeomorphic to infranilmanifolds and the structure group is affine. 
In particular, $X_i$ is a topological manifold for sufficiently large $i$.
\end{thm2}

\subsection{Basic Constructions}

To prove Theorem~B, we first apply the construction in Theorem~\ref{local split} to obtain an $\epsilon_i$-GHA map
\[
f_i: X_i \to K
\]
that is locally almost $k$-splitting. We claim that, for sufficiently large $i$, the map $f_i$ is actually a fibration. Suppose, to the contrary, that there exists a sequence $p_i \in X_i$ converging to $p \in K$, such that $f_i$ fails to be a fibration on every sufficiently small neighborhood of $p_i$. We will derive a contradiction by constructing a local fibration structure.

By the generalized Margulis lemma~\ref{KW2}, there exists $\epsilon_0>0$ such that the image
\[
G_i' = \operatorname{Im}\bigl(\pi_1(B_{\epsilon_0}(p_i),p_i) \to \pi_1(B_{\rho}(p_i),p_i)\bigr)
\]
contains a nilpotent subgroup $G_i$ of index at most $C$. Let $\widetilde{B}(p_i,\epsilon_0,\rho)$ denote a connected component of the preimage of $B_{\epsilon_0}(p_i)$ in the universal cover of $B_{\rho}(p_i)$. After passing to a subsequence, the pointed spaces satisfy
\[
(\widetilde{B}(p_i,\epsilon_0,\rho),\tilde{p}_i) \overset{pGH}{\longrightarrow} (Y,\tilde{p})
\]
by Theorem~\ref{precompact}. Moreover, by~\cite{Huang2020}, every tangent cone of $Y$ is isometric to $\mathbb{R}^N$, since each tangent cone of $K$ is $\mathbb{R}^k$.

Choose $r_i \to \infty$ sufficiently slowly, and pass to a subsequence if necessary. Then we obtain the following diagram:
\begin{equation}\label{e1}
\begin{CD}
(r_i \widetilde{B}(p_i,\epsilon_0,\rho),\tilde{p}_i,G_i,G_i')
@>eGH>>
(\mathbb{R}^N,0^N,G,G')\\
@VV\pi V @VV\pi V\\
(r_i \widetilde{B}(p_i,\epsilon_0,\rho)/G_i,p_i',G_i'/G_i)
@>GH>>
(Y',p',\bar{G})\\
@VV\pi V @VV\pi V\\
(r_i B_{\epsilon_0}(p_i),p_i)
@>F_i>>
(\mathbb{R}^k,0^k),
\end{CD}
\end{equation}
where $F_i=\mathrm{exp}_p^{-1}\circ f_i$ is almost $k$-splitting by Theorem~\ref{local split}. Since $G_i\unlhd G_i'$ with $|G_i'/G_i|\le C$, the quotient $\widetilde{B}(p_i,\epsilon_0,\rho)/G_i$ is a finite cover of $B_{\epsilon_0}(p_i)$.

By the covering Lemma~\ref{covering}, we can lift $\mathbb{R}^k$ via the submetry $\pi$,
\[
Y'=\mathbb{R}^k\times Y'', 
\qquad 
\mathbb{R}^N=\mathbb{R}^k\times\mathbb{R}^{N-k}.
\]
Thus $\bar{G}$, $G$, and $G'$ act trivially on the $\mathbb{R}^k$-factor of $Y'$ and $\mathbb{R}^N$, respectively. In particular, $Y''/(G_i'/G_i)$ is a single point. As $Y''$ is connected and $G_i'/G_i$ is finite, $Y''$ consists of a single point, and hence $Y'=\mathbb{R}^k$. It follows that $\bar{G}$ action is trivial, while $G$ and $G'$ act transitively on the $\mathbb{R}^{N-k}$-factor of $\mathbb{R}^N=\mathbb{R}^k\times\mathbb{R}^{N-k}$. Arguing as in Lemma~\ref{free}, we obtain:

\begin{lemma}\label{D}
$G$ acts freely and is isomorphic to $\mathbb{R}^{N-k}$.
\end{lemma}

By the $(\rho,v)$-bounded covering geometry and the proof of Lemma~\ref{nofix}, the sequence $\{G_i\}$ has no small subgroups. The diagram~\eqref{e1} therefore simplifies to
\begin{equation}\label{e2}
\begin{CD}
(r_i \widetilde{B}(p_i,\epsilon_0,\rho),\tilde{p}_i,G_i,G_i')
@>eGH>>
(\mathbb{R}^N,0^N,G=\mathbb{R}^{N-k},G')\\
@VV\pi V @VV\pi V\\
(r_i \widetilde{B}(p_i,\epsilon_0,\rho)/G_i,p_i',G_i'/G_i)
@>eGH>>
(\mathbb{R}^k,0^k,\mathrm{id})\\
@VV\pi V @VV\pi V\\
(r_i B_{\epsilon_0}(p_i),p_i)
@>F_i>>
(\mathbb{R}^k,0^k).
\end{CD}
\end{equation}

The proof of Theorem~B follows the same general strategy as that of Theorem~A: we use a nilpotent subgroup of the (local relative) fundamental group to construct a simply connected nilpotent Lie group, and then show that the associated covering space admits a local fibration structure. In the present setting, however, two additional difficulties arise.

First, the generators of $G_i$ may not lie in a small or bounded neighborhood with respect to the rescaled metric. Second, Theorem~\ref{local} does not directly apply to local relative fundamental groups since $\widetilde{B}(p_i,\epsilon_0,\rho)$ may not be simply connected.

From now on, we work with the rescaled metrics $r_i d_i$. To address the first issue, we apply the gap lemma and pass to a suitable subgroup of $G_i$.

For $r>0$, define
\[
G_i(\tilde{p}_i,r)
=
\{g \in G_i \mid r_i d_i(g\tilde{p}_i,\tilde{p}_i) \le r\}.
\]
By the gap Lemma~\ref{gap} and Lemma~\ref{D}, there exists a sequence $\epsilon_i\to 0$ with $\epsilon_i r_i\to\infty$ such that for $r\in\bigl[\epsilon_i,\tfrac{1}{\epsilon_i}\bigr]$,
the subgroup generated by $G_i(\tilde{p}_i,r)$ is constant and is denoted by $H_i$. The equivariant limit of $H_i$ contains a neighborhood of $G=\mathbb{R}^{N-k}$ and hence equals $G$.

Similarly, define
\[
G_i'(\tilde{p}_i,r)
=
\{g' \in G_i' \mid r_i d_i(g'\tilde{p}_i,\tilde{p}_i)\le r\}.
\]
For $r\in\bigl[2\epsilon_i,\tfrac{1}{\epsilon_i}-\epsilon_i\bigr]$,
the subgroup generated by $G_i'(\tilde{p}_i,r)$ is constant and is denoted by $H_i'$.

\begin{lemma}
For sufficiently large $i$, the subgroup $H_i$ is normal in $H_i'$ and has index at most $C$.
\end{lemma}

\begin{proof}
We have
\[
H_i=\langle G_i(\tilde{p}_i,\epsilon_i)\rangle \, , \, H_i'=\langle G_i'(\tilde{p}_i,2\epsilon_i)\rangle.
\]
To prove normality, let $h\in G_i'(\tilde{p}_i,2\epsilon_i)$ and $g \in G_i(\tilde{p}_i,\epsilon_i)$. Since $G_i\unlhd G_i'$, we have $h^{-1}gh\in G_i$. Moreover,
\[
\begin{aligned}
r_i d_i(\tilde{p}_i,h^{-1}gh\tilde{p}_i)
&= r_i d_i(h\tilde{p}_i,gh\tilde{p}_i) \\
&\le r_i d_i(h\tilde{p}_i,\tilde{p}_i)
 + r_i d_i(\tilde{p}_i,g\tilde{p}_i)
 + r_i d_i(g\tilde{p}_i,gh\tilde{p}_i) \\
&\le 5\epsilon_i.
\end{aligned}
\]
Thus $h^{-1}gh\in G_i(\tilde{p}_i,5\epsilon_i)\subset H_i$, proving normality. The index bound $|H_i'/H_i|\le C$ follows from $|G_i'/G_i|\le C$.
\end{proof}

If $g_i\in G_i\setminus H_i$, then
\[
r_i d_i(\tilde{p}_i,g_i\tilde{p}_i)>\tfrac{1}{\epsilon_i},
\]
so such elements diverge in the pointed Gromov--Hausdorff sense. Hence large balls in
\[
r_i \widetilde{B}(p_i,\epsilon_0,\rho)/G_i
\]
coincide with large balls in
\[
r_i \widetilde{B}(p_i,\epsilon_0,\rho)/H_i.
\]
Therefore we may (on large balls) identify
\[
r_i \widetilde{B}(p_i,\epsilon_0,\rho)/G_i
=
r_i \widetilde{B}(p_i,\epsilon_0,\rho)/H_i,
\]
and similarly
\[
r_i B_{\epsilon_0}(p_i)
=
r_i \widetilde{B}(p_i,\epsilon_0,\rho)/G_i'
=
r_i \widetilde{B}(p_i,\epsilon_0,\rho)/H_i'.
\]

Accordingly, diagram~\eqref{e2} becomes
\begin{equation}\label{e3}
\begin{CD}
(r_i \widetilde{B}(p_i,\epsilon_0,\rho),\tilde{p}_i,H_i,H_i')
@>eGH>>
(\mathbb{R}^N,0^N,H=\mathbb{R}^{N-k},H')\\
@VV\pi V @VV\pi V\\
(r_i \widetilde{B}(p_i,\epsilon_0,\rho)/H_i,p_i',H_i'/H_i)
@>eGH>>
(\mathbb{R}^k,0^k,\mathrm{id})\\
@VV\pi V @VV\pi V\\
(r_i B_{\epsilon_0}(p_i),p_i)
@>F_i>>
(\mathbb{R}^k,0^k).
\end{CD}
\end{equation}
Moreover, $H_i$ and $H_i'$ are both generated by short loops.

\subsection{The existence of local almost product structure}

In this subsection, we will always work with the rescaled metric $r_i d_i$ on $(X_i,p_i)$. Let $B^{r_i}_{r}(p_i)$ denote the $r$-ball centered at $p_i$ with respect to the rescaled metric $r_i d_i$. Throughout, all metric balls are taken with respect to this rescaled metric unless otherwise specified.

The goal of this subsection is to prove the following result.

\begin{lemma}[Existence of a local product structure]\label{product}
In the context of \eqref{e3}, there exists $r>0$ such that for sufficiently large $i$, the ball $B^{r_i}_r(p_i')$, after possibly removing some points near the boundary, is bi-H\"{o}lder homeomorphic to 
\[
B_r(0^k)\times \mathcal{N}_i/\hat{H}_i,
\]
where $B_r(0^k)\subset \mathbb{R}^k$ and $\mathcal{N}_i$ is a simply connected nilpotent Lie group with lattice $\hat{H}_i$.
\end{lemma}

\begin{remark}
Lemma~\ref{product} asserts the existence of a nilmanifold fibration structure on a neighborhood of $p_i'$ (not $p_i$). The obstruction to constructing an infranilmanifold fibration structure near $p_i$ is that $H_i$ is not necessarily isomorphic to the lattice $\hat{H}_i$. This issue will be addressed in Lemma~\ref{no ker}.
\end{remark}

The proof of Lemma~\ref{product} follows the same strategy as that of Theorem~A, applied to the collapsing directions $H_i \to \mathbb{R}^{N-k}$. In Theorem~A, we first used Theorem~\ref{nil} to obtain a nilprogression structure in a neighborhood of the identity in the fundamental group, and then applied Theorem~\ref{local} to conclude that the fundamental group is isomorphic to a lattice in a nilpotent Lie group. In the present situation, however, we cannot apply Theorem~\ref{local} directly to $H_i$ and the corresponding relative covers. Thus, the local structure of $H_i$ does not determine the global structure. Therefore, we instead consider the groupification of a neighborhood of the identity in $H_i$.

By Theorem~\ref{nil}, there exists $r>0$ such that for sufficiently large $i$, $H_i(\tilde{p}_i,10)$ contains a nilprogression $P_i$, which in turn contains $H_i(\tilde{p}_i,4r)$.

\begin{lemma}\label{same}
For any fixed $\epsilon$ with $0<\epsilon<4r$ and sufficiently large $i$, the groupification of $H_i(\tilde{p}_i,4r)$ is naturally isomorphic to the groupification of $H_i(\tilde{p}_i,\epsilon)$.
\end{lemma}

\begin{proof}
We have
\[
H_i(\tilde{p}_i,\epsilon)\subset H_i(\tilde{p}_i,4r)\subset P_i.
\]
Moreover, $H_i$ is generated by $H_i(\tilde{p}_i,\epsilon_i)$, where $\epsilon_i\to 0$. Using the escape norm, the generators of $P_i$ may be chosen as elements $u_j\in H_i(\tilde{p}_i,\epsilon_i)$, and the relations in $P_i$ are contained in $H_i(\tilde{p}_i,\epsilon)$. Hence the groupifications of $H_i(\tilde{p}_i,4r)$ and $H_i(\tilde{p}_i,\epsilon)$ are both naturally isomorphic to the groupification of $P_i$, which is a lattice in a simply connected nilpotent Lie group.
\end{proof}

Let $\hat{H}_i$ denote the groupification of $H_i(\tilde{p}_i,4r)$. For sufficiently large $i$, we know that $\langle H_i(\tilde{p}_i,r)\rangle=H_i$, so there is a natural surjective homomorphism
\[
s_i:\hat{H}_i\to H_i.
\]
We will apply the argument from the proof of Theorem~A to $\hat{H}_i$.

We construct the gluing space
\[
\hat{H}_i\times_{H_i(\tilde{p}_i,4r)} \bar{B}^{r_i}_{2r}(\tilde{p}_i),
\]
which is a covering space of $\bar{B}^{r_i}_{2r}(p_i)$. Let $Y_i$ denote the preimage of $\bar{B}^{r_i}_{r}(p_i)$ in this covering space. Then $Y_i$ is connected, since the generators of $\hat{H}_i$ lie in the image of the natural injective pseudo-group homomorphism
\[
H_i(\tilde{p}_i,\epsilon_i)\hookrightarrow \hat{H}_i,
\]
and $\epsilon_i \ll r$ for sufficiently large $i$.

We have $H_i=\hat{H}_i/\ker(s_i)$ and define
\[
Z_i=Y_i/\ker(s_i).
\]
Then $Z_i$ is the union of all $H_i$-orbits of $\bar{B}^{r_i}_{r}(\tilde{p}_i)$, and
\[
Z_i=\pi^{-1}(\bar{B}^{r_i}_{r}(p_i))\subset r_i\widetilde{B}(p_i,\epsilon_0,\rho).
\]
Using \eqref{e3} (ignoring $H_i'$ for the moment), together with the precompactness theorem, we obtain:

\begin{center}
\begin{align}
\begin{CD}
(Y_i,\hat{p}_i,\ker(s_i),\hat{H}_i) @>eGH>> (Y,\hat{p},S,\hat{H})\\
@VV\pi V @VV\pi V\\
(Z_i,\tilde{p}_i,H_i=\hat{H}_i/\ker(s_i)) @>eGH>> (\bar{B}_{r}(0^k)\times \mathbb{R}^{N-k},0^N,H=\mathbb{R}^{N-k})\\
@VV\pi V @VV\pi V\\
(\bar{B}^{r_i}_{r}(p_i'),p_i') @>F_i>> (\bar{B}_{r}(0^k),0^k)
\end{CD}
\end{align}
\end{center}
where $S$ is the limit group of $\ker(s_i)$ and $H=\hat{H}/S$.

\medskip

$\ker(s_i)$ corresponds to relations in $H_i$ which are not contained in a small neighborhood of $e$. We next show that nontrivial elements in $\ker(s_i)$ actually diverge to infinity.
\begin{lemma}
The group $S$ is trivial. Consequently,
\[
Y=\bar{B}_{r}(0^k)\times \mathbb{R}^{N-k}
\quad\text{and}\quad
\hat{H}=\mathbb{R}^{N-k},
\]
which acts by free translations.
\end{lemma}

\begin{proof}
By \cite{PanWang2021,Zamora2020}, for any $g \in \ker(s_i)$ we have
\[
r_i d_i(g\hat{p}_i,\hat{p}_i)\ge r.
\]
Hence the limit group $S$ is discrete.

We claim that the action of $S$ is free. Suppose, for contradiction, that $S$ has nontrivial stabilizers. Then there exist a nontrivial element $h \in S$ and a point $x \in Y$ such that $hx = x$. We may assume that $x$ is not a boundary point (otherwise we enlarge the radius slightly when defining $Y_i$). Let $\bar{x}$ be the image of $x$ in $Y/S$. The tangent cone at $\bar{x}$ is $\mathbb{R}^N$, hence the tangent cone at $x$ is also $\mathbb{R}^N$. However, $\langle h\rangle$ induces a nontrivial limit action on the tangent cone at $x$ by volume convergence and no small subgroups. Therefore $x$ and $\bar{x}$ cannot have the same tangent cone. This is a contradiction.

Therefore $S$ acts freely and discretely. Since $\bar{B}_{r}(0^k)\times \mathbb{R}^{N-k}$ is simply connected, $S$ must be trivial.
\end{proof}

We thus obtain $\hat{H}=H=\mathbb{R}^{N-k}$ and the simplified diagram:
\begin{center}
\begin{align}\label{e4}
\begin{CD}
(Y_i,\hat{p}_i,\ker(s_i),\hat{H}_i) @>eGH>> (Y=\bar{B}_{r}(0^k)\times \mathbb{R}^{N-k},0^N,\mathrm{id},\hat{H}=\mathbb{R}^{N-k})\\
@VV\pi V @VV\pi V\\
(Z_i,\tilde{p}_i,H_i=\hat{H}_i/\ker(s_i)) @>eGH>> (\bar{B}_{r}(0^k)\times \mathbb{R}^{N-k},0^N,H=\mathbb{R}^{N-k})\\
@VV\pi V @VV\pi V\\
(\bar{B}^{r_i}_{r}(p_i'),p_i') @>F_i>> (\bar{B}_{r}(0^k),0^k)
\end{CD}
\end{align}
\end{center}

\medskip

Now we can prove Lemma~\ref{product}. The idea is to apply the proof of Theorem A to the convergence $\hat{H}_i \to \mathbb{R}^{N-k}$. 
\begin{proof}[Proof of Lemma~\ref{product}]
By Theorem~\ref{nil} and diagram~\eqref{e4}, there exists $\epsilon>0$ such that for sufficiently large $i$ we can find a nilprogression
\[
P_i'\subset \hat{H}_i(\hat{p}_i,1)
\]
containing
\[
\hat{H}_i(\hat{p}_i,\epsilon)
=\{g \in \hat{H}_i\mid r_i d_i(g\hat{p}_i,\hat{p}_i) \le \epsilon\}.
\]
Assume $\epsilon<4r$. Since nontrivial elements in $\ker(s_i)$ diverge to infinity, the map
\[
s_i: \hat{H}_i (\tilde{p}_i,\epsilon)\to H_i(\hat{p}_i,\epsilon)
\]
is injective, hence a pseudo-group isomorphism. By Lemma~\ref{same}, the groupification of $H_i(\tilde{p}_i,\epsilon)$ is isomorphic to $\hat{H}_i$, and therefore $\hat{H}_i$ is isomorphic to the groupification of $\hat{H}_i(\hat{p}_i,\epsilon)$, which coincides with the groupification of $P_i'$.

Using the nilprogression $P_i'$, we construct an $(N-k)$-dimensional nilpotent Lie group $\mathcal{N}_i$ and endow it with a left-invariant metric satisfying Lemma~\ref{metric}. Then $\hat{H}_i$ is isomorphic to a lattice in $\mathcal{N}_i$.

Let $(0^k,e)\in \bar{B}_{r}(0^k) \times \mathcal{N}_i$. By Lemma~\ref{localeg}, we construct a map
\[
h_i': B^{r_i}_{r}(\hat{p}_i)\to B_{r}(0^k,e)
\]
which is an $\epsilon_i$-GHA and $\epsilon_i$-almost $\hat{H}_i(\hat{p}_i,r)$-equivariant, with $\epsilon_i\to 0$.

By Lemma~\ref{extension}, possibly after removing points near the boundary, we extend $h_i'$ to a global map
\[
h_i:(Y_i,\hat{p}_i) \to \bar{B}_{r}(0^k)\times \mathcal{N}_i
\]
which is an $\epsilon_i$-GHA on any $r$-ball and $\epsilon_i$-almost $\hat{H}_i$-equivariant.

Next, choose a normal subgroup $\hat{H}_i'\unlhd \hat{H}_i$ of finite index such that $\hat{H}_i'\cap B_r(e)=\{e\}$. By Theorem~\ref{inv}, after possibly further blowing up the metric and removing boundary points, there exists a $\hat{H}_i/\hat{H}_i'$-equivariant $\epsilon_i$-GHA
\[
f_{\hat{H}_i/\hat{H}_i'}:
Y_i/\hat{H}_i'\longrightarrow
\bar{B}_{r}(0^k)\times \mathcal{N}_i/\hat{H}_i'
\]
which is locally almost $N$-splitting. In particular, $f_{\hat{H}_i/\hat{H}_i'}$ is bi-H\"{o}lder. Passing to the quotient by $H_i=\hat{H}_i/\ker(s_i)$, we conclude that
\[
Y_i/H_i=\bar{B}^{r_i}_{r}(p_i')
\]
is bi-H\"{o}lder homeomorphic to
\[
\bar{B}_{r}(0^k)\times \mathcal{N}_i/\hat{H}_i.
\]
\end{proof}

\subsection{Local Relative Fundamental Groups}

Lemma~\ref{product} does not immediately imply Theorem~B, since $H_i$ need not be isomorphic to $\hat{H}_i$. Equivalently, the kernel $\ker(s_i)$ in~\eqref{e4} may be nontrivial, or the nilpotency rank of $H_i$ may be strictly less than $N-k$. In particular, the structure of $H_i'$ is not yet determined. We resolve this issue by passing to a further contradiction sequence.

The goal of this subsection is the following.

\begin{lemma}\label{no ker}
After passing to another contradiction sequence if necessary, we may assume that $\hat{H}_i = H_i$ and that $H_i$ is a normal subgroup of $H_i'$ with index $\le C(N)$.
\end{lemma}

The strategy to prove Lemma \ref{no ker} is to apply Lemma~\ref{product}. We have proved $H_i$ is locally isomorphic to  $\hat{H}_i$. If they are not equal, then there is a relation among the generators of $H_i$ which cannot be observed locally since this relation does not hold in $\hat{H}_i$. Then we use the local product structure in Lemma~\ref{product} to show that, after taking smaller balls and possibly passing to a subsequence, this relation can be eliminated. Therefore $H_i$ is isomorphic to the groupification $\hat{H}_i$. 

From~\eqref{e4} we obtain the diagram
\begin{equation}\label{e5}
\begin{CD}
(Y_i,\hat{p}_i,\ker(s_i),\hat{H}_i)
@>eGH>>
(Y=\bar{B}_{r}(0^k)\times \mathbb{R}^{N-k},0^N,\mathrm{id},\hat{H}=\mathbb{R}^{N-k})\\
@VV\pi V @VV\pi V\\
(Z_i,\tilde{p}_i,H_i=\hat{H}_i/\ker(s_i))
@>GH>>
(\bar{B}_{r}(0^k)\times \mathbb{R}^{N-k},0^N,H=\mathbb{R}^{N-k})\\
@VV\pi V @VV\pi V\\
(\bar{B}^{r_i}_{r}(p_i'),p_i')
@>eGH>>
(\bar{B}_{r}(0^k),0^k),
\end{CD}
\end{equation}
and $B^{r_i}_{r}(p_i')$ is bi-H\"older homeomorphic to
\[
B_r(0^k)\times \mathcal{N}_i/\hat{H}_i.
\]

Let $C$ denote the constant in the generalized Margulis lemma. The quotient
\[
r_i \widetilde{B}(p_i,\epsilon_0,\rho)/G_i
\]
is a finite cover of $r_i B_{\epsilon_0}(p_i)$ of degree $\le C$. For any $g \in G_i\setminus H_i$ we have
\[
r_i d_i(g\tilde{p}_i,\tilde{p}_i)\ge \frac{1}{\epsilon_i},
\]
so $r_i \widetilde{B}(p_i,\epsilon_0,\rho)/G_i$ agrees with
$r_i \widetilde{B}(p_i,\epsilon_0,\rho)/H_i$ on large balls. Hence, for any fixed $s>0$ and sufficiently large $i$, a connected component of the preimage of $B^{r_i}_s(p_i)$ in
$r_i \widetilde{B}(p_i,\epsilon_0,\rho)/H_i$
is a cover of $B^{r_i}_s(p_i)$ with index $\le C$; in particular, it is contained in a $Cs$-ball.

Let
\[
\widetilde{H}_i'
=
\operatorname{Im}\!\Bigl(
\pi_1\bigl(B^{r_i}_{\frac{r}{40C^2}}(p_i),p_i\bigr)
\to
\pi_1\bigl(B^{r_i}_{\frac{r}{10C}}(p_i),p_i\bigr)
\Bigr).
\]
Let $T_i$ be a connected component of the preimage of
$B^{r_i}_{\frac{r}{40C^2}}(p_i)$
in
$r_i \widetilde{B}(p_i,\epsilon_0,\rho)/H_i$.
Then $T_i$ is a finite cover of
$B^{r_i}_{\frac{r}{40C^2}}(p_i)$
of degree at most $C$, and in particular
\[
T_i \subset B^{r_i}_{\frac{r}{20C}}(p_i').
\]

\begin{lemma}
There exists a natural injective homomorphism
\[
i:\hat{H}_i \longrightarrow \widetilde{H}_i'.
\]
\end{lemma}

\begin{proof}
By Lemma~\ref{product}, $B^{r_i}_{r}(p_i')$ is bi-H\"older homeomorphic to
$B_r(0^k)\times \mathcal{N}_i/\hat{H}_i$, and
$\operatorname{diam}(\mathcal{N}_i/\hat{H}_i)<\epsilon_i\to 0$.
In particular, the image of
\[
\pi_1\bigl(B^{r_i}_{2\epsilon_i}(p_i'),p_i'\bigr)
\to
\pi_1\bigl(B^{r_i}_{\frac{r}{10}}(p_i'),p_i'\bigr)
\]
is isomorphic to $\hat{H}_i$.

For $g\in\hat{H}_i$, choose a loop
$\gamma_g\subset B^{r_i}_{2\epsilon_i}(p_i')$
representing $g$ in this image. Let
\[
\pi:B^{r_i}_{2\epsilon_i}(p_i')\to B^{r_i}_{2\epsilon_i}(p_i)
\]
be the projection. Define $i(g)$ to be the element of
$\widetilde{H}_i'$ represented by the loop $\pi(\gamma_g)$.

To see that $i$ is well defined, suppose $\gamma_1$ and $\gamma_2$
are homotopic in $B^{r_i}_{\frac{r}{10}}(p_i')$.
By the local product structure in Lemma~\ref{product},
they are homotopic in
$B^{r_i}_{\frac{r}{10C}}(p_i')$.
Projecting the homotopy shows that
$\pi(\gamma_1)$ and $\pi(\gamma_2)$
are homotopic in
$B^{r_i}_{\frac{r}{10C}}(p_i)$,
and hence represent the same element of $\widetilde{H}_i'$.
Thus $i$ is well defined and clearly a homomorphism.

To prove injectivity, suppose $g\in\hat{H}_i$ is represented by
$\gamma_g\subset B^{r_i}_{2\epsilon_i}(p_i')$,
and that $\pi(\gamma_g)$ is contractible in
$B^{r_i}_{\frac{r}{10C}}(p_i)$.
A connected component of the preimage of
$B^{r_i}_{\frac{r}{10C}}(p_i)$
in $r_i \widetilde{B}(p_i,\epsilon_0,\rho)/H_i$ is contained in
$B^{r_i}_{\frac{r}{10}}(p_i')$.
By the homotopy lifting property,
$\gamma_g$ is contractible in
$B^{r_i}_{\frac{r}{10}}(p_i')$,
so $g$ is trivial in the image of
\[
\pi_1\bigl(B^{r_i}_{2\epsilon_i}(p_i'),p_i'\bigr)
\to
\pi_1\bigl(B^{r_i}_{\frac{r}{10}}(p_i'),p_i'\bigr).
\]
Hence $i$ is injective.
\end{proof}

\begin{lemma}\label{index}
The index $|\widetilde{H}_i' : i(\hat{H}_i)| \le C$.
\end{lemma}

\begin{proof}
Recall that $T_i$ is a finite cover of $B^{r_i}_{\frac{r}{40C^2}}(p_i)$
of degree at most $C$. Suppose that $|\widetilde{H}_i' : i(\hat{H}_i)| > C.$
Then there exist loops $\gamma_1,\gamma_2$ in
$B^{r_i}_{\frac{r}{40C^2}}(p_i)$ lying in different cosets of $\widetilde{H}_i' / i(\hat{H}_i)$, while their lifts in $T_i$, denoted $\gamma_1'$ and $\gamma_2'$, have the same endpoints.

Since $T_i \subset B^{r_i}_{\frac{r}{20C}}(p_i')$ and $B^{r_i}_{r}(p_i')$ is bi-H\"older homeomorphic to
$B_r(0^k)\times \mathcal{N}_i/\hat{H}_i$, the loop $\gamma_1'(\gamma_2')^{-1}$ is homotopic to a loop $\gamma_g$ corresponding to some $g \in \hat{H}_i$. Moreover, the homotopy can be chosen to lie entirely within
$B^{r_i}_{\frac{r}{10C}}(p_i')$, and the loop $\gamma_g$ is contained in $B^{r_i}_{2\epsilon_i}(p_i')$ since $\operatorname{diam}(\mathcal{N}_i/\hat{H}_i)\to 0$.

Projecting this homotopy to
$B^{r_i}_{\frac{r}{10C}}(p_i)$ shows that $\gamma_1\gamma_2^{-1}$ represents the element $i(g)\in i(\hat{H}_i)$, which is a contradiction to that $\gamma_1,\gamma_2$ lie in different cosets of $\widetilde{H}_i' / i(\hat{H}_i)$.
\end{proof}

Since $\widetilde{H}_i'$ contains the nilpotent subgroup
$i(\hat{H}_i)$ of index $\le C$,
Lemma~\ref{normal} implies that
$\widetilde{H}_i'$ contains a normal nilpotent subgroup
$\widetilde{H}_i$ of index $\le C(N)=C!$.
Because $\hat{H}_i$ is torsion-free nilpotent of rank $N-k$
and $\widetilde{H}_i$ has finite index in $i(\hat{H}_i)$,
we have $\operatorname{rank}(\widetilde{H}_i)=N-k$.

\begin{proof}[Proof of Lemma~\ref{no ker}]
Return to the setup of~\eqref{e1}.
Recall that $\widetilde{H}_i'$ is the image of
\[
\pi_1\bigl(B^{r_i}_{\frac{r}{40C^2}}(p_i),p_i\bigr)
\to
\pi_1\bigl(B^{r_i}_{\frac{r}{10C}}(p_i),p_i\bigr).
\]
Choose $r_i'\to\infty$ sufficiently slowly.
Then we obtain another contradiction sequence:
\begin{equation}\label{e6}
\begin{CD}
(r_i'r_i \widetilde{B}(p_i,\epsilon_0,\rho),\tilde{p}_i,
\widetilde{H}_i,\widetilde{H}_i')
@>eGH>>
(\mathbb{R}^N,0^N,G,G')\\
@VV\pi V @VV\pi V\\
(r_i'r_i \widetilde{B}(p_i,\epsilon_0,\rho)/\widetilde{H}_i,
p_i',\widetilde{H}_i'/\widetilde{H}_i)
@>GH>>
(Y',p',\bar{G})\\
@VV\pi V @VV\pi V\\
(r_i'r_i B_{\epsilon_0}(p_i),p_i)
@>F_i>>
(\mathbb{R}^k,0^k).
\end{CD}
\end{equation}
Elements outside $\widetilde{H}_i'$ diverge and thus disappear in the limit, which justifies replacing $G_i'$ by $\widetilde{H}_i'$.

Since $\widetilde{H}_i$ has index $\le C(N)$ in $i(\hat{H}_i)$,
it is generated by short elements
\[
\widetilde{H}_i(\hat{p}_i,\epsilon_i)
=
\{g \in \widetilde{H}_i \mid r_i d_i(g\hat{p}_i,\hat{p}_i)\le \epsilon_i\}.
\]
Using the same construction as for $H_i$ and $H_i'$ in~\eqref{e2},
we define $\bar{H}_i$ and $\bar{H}_i'$ from~\eqref{e6} via the gap lemma.
Choosing $r_i'\epsilon_i\to 0$, we obtain $\bar{H}_i=\widetilde{H}_i$.
Since $\bar{H}_i'\subset \widetilde{H}_i'$,
the subgroup $\bar{H}_i$ is normal in $\bar{H}_i'$ of index $\le C(N)$.

Repeating the groupification construction for $\bar{H}_i$,
we obtain $\hat{\bar{H}}_i$ together with a natural surjective homomorphism
\[
s_i:\hat{\bar{H}}_i\to \widetilde{H}_i.
\]
Because both $\hat{\bar{H}}_i$ and $\widetilde{H}_i$ have nilpotency rank $N-k$,
the kernel of $s_i$ must be trivial.
Hence $\widetilde{H}_i$ is isomorphic to $\hat{\bar{H}}_i$.

Working with diagram~\eqref{e6} and replacing groups if necessary,
we may therefore assume that $H_i$ is a normal subgroup of $H_i'$ of index $\le C(N)$ and that $\hat{H}_i=H_i$.
\end{proof}

\subsection{Proof of Theorem B}

We now prove that $f_i:X_i\to K$ is a fibration map. 
Compared with the proof of Lemma~\ref{product}, there are two main differences in the present argument. 
First, by Lemma~\ref{no ker}, the group $H_i$ is isomorphic to a lattice; hence $H_i'$ can be identified with a discrete subgroup of $\mathcal{N}_i\rtimes \mathrm{Aut}(\mathcal{N}_i)$ by a similar argument to that used in the proof of Theorem A. 
This allows us to construct an infranilmanifold fiber in $X_i$. 
Second, we apply the gluing argument in Theorem~\ref{inv} more carefully to ensure that 
\[
F_i=\mathrm{exp}_p^{-1}\circ f_i
\]
is itself a fibration map near $p$, and therefore $f_i$ is a fibration.

\begin{proof}[Proof of Theorem~B: existence of a fibration with infranilmanifold fibers]

Consider diagram~\eqref{e3}:
\[
\begin{CD}
(r_i \widetilde{B}(p_i,\epsilon_0,\rho),\tilde{p}_i,H_i,H_i')
@>eGH>>
(\mathbb{R}^N,0^N,H=\mathbb{R}^{N-k},H')\\
@VV\pi V @VV\pi V\\
(r_i \widetilde{B}(p_i,\epsilon_0,\rho)/H_i',p_i',H_i'/H_i)
@>GH>>
(\mathbb{R}^k,0^k,\mathrm{id})\\
@VV\pi V @VV\pi V\\
(r_iB_{\epsilon_0}(p_i),p_i)
@>F_i>>
(\mathbb{R}^k,0^k).
\end{CD}
\]
Assume, for contradiction, that $F_i = \mathrm{exp}_p^{-1} \circ f_i$ is not a fibration near $p$. We will show that for sufficiently large $i$, $F_i$ is in fact a fibration, yielding a contradiction.

Apply Lemma~\ref{no ker} to diagram~\eqref{e7}:
\begin{equation}\label{e7}
\begin{CD}
(Z_i,\tilde{p}_i,H_i=\hat{H}_i,H_i')
@>eGH>>
(\bar{B}_r(0^k)\times \mathbb{R}^{N-k},0^N,H=\mathbb{R}^{N-k},H')\\
@VV\pi V @VV\pi V\\
(\bar{B}^{r_i}_r(p_i'),p_i',H_i'/H_i)
@>eGH>>
(\bar{B}_r(0^k),0^k,\mathrm{id})\\
@VV\pi V @VV\pi V\\
(\bar{B}^{r_i}_r(p_i),p_i)
@>F_i>>
(\bar{B}_r(0^k),0^k).
\end{CD}
\end{equation}

By Lemma~\ref{no ker}, $H_i=\hat{H}_i$ is a lattice in the simply connected nilpotent Lie group $\mathcal{N}_i$. 
Using the arguments in Lemmas~\ref{localeg}, ~\ref{extension} and \ref{G_i'-inv}, we construct a global map
\[
\tilde{h}_i:Z_i \to \bar{B}_r(0^k)\times \mathcal{N}_i
\]
which is an $\epsilon_i$-GHA on every $r$-ball and is $\epsilon_i$-almost $H_i'$-equivariant.

Choose a normal subgroup $H_i'' \unlhd H_i'$ of finite index such that
$B_r(e)\cap H_i''=\{e\}$, where $e\in \mathcal{N}_i$.
Then $(Z_i/H_i'',H_i'/H_i'')$ is equivariantly Gromov--Hausdorff close on each $r$-ball to
\[
(\bar{B}_r(0^k)\times \mathcal{N}_i/H_i'',H_i'/H_i'').
\]
Applying Theorem~\ref{inv}, we obtain an $H_i'/H_i''$-equivariant homeomorphism
\[
Z_i/H_i'' \longrightarrow \bar{B}_r(0^k)\times \mathcal{N}_i/H_i'',
\]
possibly after removing points near the boundary. 
Consequently,
\[
Z_i/H_i'=\bar{B}^{r_i}_r(p_i)
\]
is bi-H\"older homeomorphic to
\[
\bar{B}_r(0^k)\times \mathcal{N}_i/H_i'.
\]

Recall that
\[
\tilde{F}_i:
r_i \widetilde{B}(p_i,\epsilon_0,\rho)
\overset{\pi}{\longrightarrow}
r_i B_{\epsilon_0}(p_i)
\overset{F_i}{\longrightarrow}
\mathbb{R}^k
\]
is an almost $k$-splitting map.
In the construction of Theorem~\ref{inv}, local almost $N$-splitting maps and group orbits are glued together via a center-of-mass argument. 
At each step, we choose the local almost $N$-splitting on $Z_i/H_i''$ so that its $\mathbb{R}^k$-component is exactly $\tilde{F}_i$. 
Since different local splittings and $H_i'$-orbits agree on the $\tilde{F}_i$-value at each point, the gluing procedure preserves the $\mathbb{R}^k$-component and modifies only the $\mathcal{N}_i$-component.

Therefore,
\[
F_i: B_r(p_i)
\longrightarrow
\mathbb{R}^k
\]
is precisely the projection onto the $\mathbb{R}^k$-factor of the bi-H\"{o}lder homeomorphism
\[
B_r(p_i)
\overset{\mathrm{homeo}}{\longrightarrow}
B_r(0^k)\times \mathcal{N}_i/H_i'
\overset{\pi}{\longrightarrow}
B_r(0^k).
\]
Hence $F_i$ is a fibration map with infranilmanifold fibers, completing the proof.
\end{proof}

\medskip

We next show that the nilpotent structure of $\mathcal{N}_i$ described in Theorem~B is independent of the choice of base point in $X_i$. 

The nilpotent structure of $\mathcal{N}_i$ is determined by the escape norm on
\[
A=H_i(\tilde{p}_i,r).
\]
The group $\mathcal{N}_i$ is constructed from a nilprogression
\[
P(u_1,\dots,u_{N-k};C_1,\dots,C_{N-k})
\]
with escape norms ordered by
\[
\|u_1\|_A \le \|u_2\|_A \le \cdots \le \|u_{N-k}\|_A < \epsilon_i,
\qquad \epsilon_i \to 0.
\]
$H_i=\hat{H}_i$ is a lattice in $\mathcal{N}_i$.

After passing to a subsequence, we may assume that
\[
\frac{\|u_{j+1}\|_A}{\|u_j\|_A}
\]
converges to a finite number or to $\infty$ as $i\to\infty$.
Hence there exist integers
\[
1=l_1<l_2<\dots<l_L\le N-k
\]
such that, for sufficiently large $i$,
\[
\frac{\|u_j\|_A}{\|u_l\|_A}\le C
\quad\text{whenever}\quad
l_s\le l<j<l_{s+1},
\]
and
\[
\frac{\|u_j\|_A}{\|u_l\|_A}\to\infty
\quad\text{whenever}\quad
l<l_s\le j.
\]

Fix $i$ sufficiently large and set
\[
G_s=\langle u_1,\dots,u_{l_{s+1} - 1}\rangle \subset H_i,
\qquad 1\le s\le L.
\]
By Theorem~\ref{escape}, for any $l_{s} \le j_1,j_2 < l_{s+1}$, we have
\[
\|[u_{j_1},u_{j_2}]\|_A \ll \min\{\|[u_{j_1}]\|_A,\|[u_{j_2}]\|_A\},
\]
and hence $[u_{j_1},u_{j_2}] \in G_{s-1}$. It follows that $G_s/G_{s-1}$ lies in the center of $H_i/G_{s-1}$.

Let $\mathcal{N}_{i,s}$ be the simply connected subgroup of $\mathcal{N}_i$ with lattice $G_s$.
Then
\[
(\mathcal{N}_i/\mathcal{N}_{i,s-1})/(G_s/G_{s-1})
\]
is a torus bundle over $\mathcal{N}_i/\mathcal{N}_{i,s}$.

We now show that this construction is independent of the base point.
Consider diagram~\eqref{e3} and take
$\tilde{x}_i,\tilde{y}_i\in B^{r_i}_1(\tilde{p}_i)$.
For $0<r<1/100$, define
\[
A_1=\{g \in H_i\mid r_i d_i(g\tilde{x}_i,\tilde{x}_i)\le r\},
\qquad
A_2=\{g \in H_i\mid r_i d_i(g\tilde{y}_i,\tilde{y}_i)\le r\}.
\]
We claim that for $g \in A_1 \cap A_2$ and $i$ large,
\[
\frac12 \le \frac{\|g\|_{A_1}}{\|g\|_{A_2}} \le 2.
\]

Indeed, by definition,
$g^{1/\|g\|_{A_1}}\notin A_1$, so
$r_id_i (g^{1/\|g\|_{A_1}}\tilde{x}_i,\tilde{x}_i)\ge r$.
Since the equivariant limit of $H_i$ is the translation group
$\mathbb{R}^{N-k}$ (Lemma~\ref{D}),
we obtain
$r_id_i(g^{1/\|g\|_{A_1}}\tilde{y}_i,\tilde{y}_i)\ge \frac{2r}{3}$ for large $i$.
Hence
$r_i d_i(g^{2/\|g\|_{A_1}}\tilde{y}_i,\tilde{y}_i)> r$,
which implies $g^{2/\|g\|_{A_1}} \notin A_2$. Therefore
\[
\|g\|_{A_2}\ge \frac{\|g\|_{A_1}}{2}.
\]
The reverse inequality is obtained similarly.

Thus the escape norms defined using different base points in a small neighborhood are uniformly equivalent, and hence determine the same nilpotent structure of $H_i$.
By the connectedness of $X_i$ and continuity of the escape norm under small base-point changes, the nilpotent structure of the fiber is independent of the base point on $X_i$ for sufficiently large $i$.
This implies that the structure group in Theorem~B can be taken to be affine, as in the gluing arguments of~\cite{CFG,Rong2022}.






\section{Limit of \texorpdfstring{$\mathrm{RCD}(-(N-1),N)$}{RCD(-(N-1),N)} Spaces with Bounded Covering Geometry}

In this section, we consider a sequence of pointed $\mathrm{RCD}(-(N-1),N)$ spaces $(X_i,d_i,\mathcal{H}^N,p_i)$ with $(\rho,v)$-bounded covering geometry and assume the following convergence:
\begin{equation}\label{e8}
(X_i,d_i,\mathcal{H}^N,p_i) \overset{\mathrm{pmGH}}\longrightarrow (X,d,\mathfrak{m},p).
\end{equation}
Assume that the rectifiable dimension of $X$ is $k$. Our goal is to show that every $k$-regular point in $X$ has a neighborhood homeomorphic to an open subset of $\mathbb{R}^k$.

\begin{theorem}\label{mani}
In the setting of \eqref{e8}, assume that $x \in X$ is a $k$-regular point. Then there exists a neighborhood of $x$ which is bi-H\"older homeomorphic to an open ball in $\mathbb{R}^k$.
\end{theorem}

Let $x_i \in X_i$ converge to $x \in X$. Let $G_i$ be the image of the natural homomorphism
\[
\pi_1(B_{\frac{\rho}{2}}(x_i),x_i) \longrightarrow \pi_1(B_{\rho}(x_i),x_i).
\]
Let $\widetilde{B}(x_i, \frac{\rho}{2}, \rho)$ denote a connected component of the preimage of $B_{\frac{\rho}{2}}(x_i)$ in the universal cover of $B_{\rho}(x_i)$. Choose $\tilde{x}_i \in \widetilde{B}(x_i,\frac{\rho}{2},\rho)$ to be a lift of $x_i$. Passing to a subsequence if necessary, we obtain equivariant Gromov--Hausdorff convergence:
\begin{center}
\begin{align}\label{e9}
\begin{CD}
(\widetilde{B}(x_i,\tfrac{\rho}{2},\rho),\tilde{x}_i,G_i) @>eGH>> (Y,\tilde{x},G) \\
@VV\pi V @VV\pi V\\
(B_{\frac{\rho}{2}}(x_i),x_i) @>GH>> (\bar{B}_{\frac{\rho}{2}}(x),x).
\end{CD}
\end{align}
\end{center}
Since $x$ is a $k$-regular point,  it follows that $\tilde{x}$ is an $N$-regular point by \cite{Huang2020}.

We next show that an almost $k$-splitting map near $x$ lifts to an almost $k$-splitting map near $\tilde{x}$.

\begin{lemma}\label{lift}
In the setting of \eqref{e8}, there exists $C(N)>0$ such that the following holds. Suppose that for some small $\epsilon,r>0$ there exists a harmonic $(k,\epsilon)$-splitting map
\[
f:B_r(x)\to\mathbb{R}^k.
\]
Define
\[
\tilde{f}:B_{\frac{r}{2}}(\tilde{x})\to\mathbb{R}^k,
\qquad
\tilde{f}=f\circ\pi.
\]
Then $\tilde{f}$ is a harmonic and $(k,C\epsilon)$-splitting map.
\end{lemma}

\begin{proof}
By Corollary 4.12 in \cite{AH18}, there exist harmonic maps
\[
f_i:B_{\frac{2r}{3}}(x_i)\to\mathbb{R}^k
\]
converging to $f|_{B_{\frac{2r}{3}}(x)}$ in the $H^{1,2}$ sense. In particular, for sufficiently large $i$, the maps $f_i$ are $(k,2\epsilon)$-splitting maps.

Define
\[
\tilde{f}_i:B_{\frac{2r}{3}}(\tilde{x}_i)\to\mathbb{R}^k,
\qquad
\tilde{f}_i=f_i\circ\pi.
\]
By the covering lemma (Lemma~\ref{covering}), each $\tilde{f}_i$ is harmonic and $(k,C\epsilon)$-splitting. After passing to a subsequence, Theorem 4.4 in \cite{AH18} implies that $\tilde{f}_i$ converges in $H^{1,2}$ to a harmonic map $\tilde{g}$ on $B_{\frac{r}{2}}(\tilde{x})$, which is also $(k,C\epsilon)$-splitting. By equivariant convergence, $\tilde{g}=f \circ\pi$ on $B_{\frac{r}{2}}(\tilde{x})$, hence $\tilde{f}=\tilde{g}$ is harmonic and $(k,C\epsilon)$-splitting.
\end{proof}

We now establish the geometric transformation theorem at regular points of $X$.

\begin{lemma}\label{trans}
In the setting of \eqref{e8}, for any $\delta>0$ there exists $\epsilon>0$ such that the following holds. Suppose that for some $r_0\le \rho$ there exists a $(k,\epsilon)$-splitting map
\[
f:B_{r_0}(x)\to\mathbb{R}^k.
\]
Then for any $s\le\frac{r_0}{10}$ there exists a $k\times k$ lower triangular matrix $T_s$ such that
\[
T_s f:B_s(x)\to\mathbb{R}^k
\]
is a $(k,\delta)$-splitting map and $|T_s|\le s^{-\delta}$.
\end{lemma}

\begin{proof}
Choose $\epsilon>0$ sufficiently small. 
By Lemma~\ref{lift}, $\tilde{f}=f\circ\pi$ is a $(k,C\epsilon)$-splitting map.

Because the rectifiable dimension of $X$ is $k$, \cite{Kitabeppu2019} shows that no tangent cone of $X$ splits off an $\mathbb{R}^{k+1}$ factor. Hence $B_{r_0}(x)$ is $r_0\Phi(\epsilon \,|\, N,k)$-close to a ball in $\mathbb{R}^k$. By \cite{Huang2020}, $B_{r_0}(\tilde{x})$ is $r_0\Phi(\epsilon \,|\, N,k,v)$-close to a ball in $\mathbb{R}^N$.

Therefore, we can find
\[
\tilde{h}:B_{\frac{r_0}{2}}(\tilde{x})\to\mathbb{R}^{N-k}
\]
such that
\[
(\tilde{f},\tilde{h}):B_{\frac{r_0}{2}}(\tilde{x})\to\mathbb{R}^N
\]
is an $(N,\Phi(\epsilon \,|\, N,k,v))$-splitting map. Applying the transformation theorem (Theorem~\ref{transformation}) to $(\tilde{f},\tilde{h})$, for sufficiently small $\epsilon$ and any $s\le\frac{r_0}{10}$ there exists an $N\times N$ lower triangular matrix $\tilde{T}_s$ such that
\[
\tilde{T}_s(\tilde{f},\tilde{h}):B_{3s}(\tilde{x})\to\mathbb{R}^N
\]
is an $(N,\delta)$-splitting map and $|\tilde{T}_s|\le s^{-\delta}$.

Let $T_s$ be the upper-left $k\times k$ submatrix of $\tilde{T}_s$. Then $T_s$ is lower triangular and $T_s\tilde{f}$ is $(k,\delta)$-splitting on $B_{3s}(\tilde{x})$. By an argument analogous to Lemma~\ref{lift}, $T_s f$ is $(k,C\delta)$-splitting on $B_s(x)$. Moreover,
\[
|T_s|\le|\tilde{T}_s|\le s^{-\delta}.\qedhere
\]
\end{proof}

\begin{proof}[Proof of Theorem~\ref{mani}]
The argument follows from the canonical Reifenberg theorem in \cite{CheegerNaberJiang2021,HondaPeng2024}. Fix $\delta>0$ small and let $\epsilon$ be given by Lemma~\ref{trans}. Since $x$ is regular, there exist $r_0>0$ and a $(k,\epsilon)$-splitting map
\[
f:B_{4r_0}(x)\to\mathbb{R}^k.
\]
For any $y\in B_{2r_0}(x)$, the restriction $f:B_{r_0}(y)\to\mathbb{R}^k$ remains a $(k,\epsilon)$-splitting map. By Lemma~\ref{trans}, for any $s\le\frac{r_0}{10}$ there exists a lower triangular matrix $T_{s,y}$ such that
\[
T_{s,y}f:B_s(y)\to\mathbb{R}^k
\]
is $(k,\delta)$-splitting with $|T_{s,y}|\le s^{-\delta}$.

We claim that $f$ is bi-H\"older on $B_{r_0}(x)$. For $y_1,y_2\in B_{r_0}(x)$, set $s=d(y_1,y_2)$. Since $T_{s,y_1}f$ is $(k,\delta)$-splitting and $\dim X=k$, it is a $\Phi(\delta\mid N,k)s$-GHA map. Hence
\[
d(T_{s,y_1}f(y_1),T_{s,y_1}f(y_2))
\ge (1-\Phi(\delta\mid N,k))\, d(y_1,y_2).
\]
Using $|T_{s,y_1}|\le s^{-\delta}$ and $s=d(y_1,y_2)$, we obtain
\[
d(f(y_1),f(y_2))
\ge s^{\delta}\, d(T_{s,y_1}f(y_1),T_{s,y_1}f(y_2))
\ge (1-\Phi(\delta\mid N,k))\, d(y_1,y_2)^{1+\delta}.
\]

On the other hand, since $f$ is harmonic, $|\nabla f^a|\le 1+\Phi(\delta\mid N,k)$ for each component $f^a$. Therefore,
\[
d(f(y_1),f(y_2))
\le (1+\Phi(\delta \mid N,k))\, d(y_1,y_2).
\]
These two estimates yield the desired bi-H\"{o}lder estimate, completing the proof.
\end{proof}

\begin{proof}[Proof of Theorem~C]
Part (a) is precisely Theorem~\ref{mani}, while part (b) follows from the construction in Theorem~B.
\end{proof}

\bibliographystyle{plain} 
\bibliography{bib}
\end{document}